\documentclass[a4paper,fleqn]{cas-sc}

\usepackage[numbers,sort&compress]{natbib}

\usepackage{amsmath}
\usepackage{amssymb}
\usepackage{amsfonts}
\usepackage{amsthm}
\usepackage{lyu}
\usepackage{xcolor}
\usepackage{booktabs}
\usepackage{algorithm}
\usepackage{algorithmic}
\usepackage{placeins}

\newtheorem{theorem}{Theorem}[section]

\newtheorem{lemma}[theorem]{Lemma}
\theoremstyle{definition}

\newtheorem{proposition}[theorem]{Proposition}
\newtheorem{remark}[theorem]{Remark}

\begin{document}
\let\WriteBookmarks\relax
\def\floatpagepagefraction{1}
\def\textpagefraction{.001}

\shorttitle{High-dimensional enhanced sampling via path-dependent McKean--Vlasov dynamics}
\shortauthors{L.~Lyu, S.~Guo and H.~Lei}

\title[mode=title]{High-Dimensional Enhanced Sampling via Regularized Path-Dependent McKean--Vlasov Dynamics using Tensor Density Approximation}

\author[1,2,3]{Liyao Lyu}[orcid=0000-0003-2994-1464]
\ead{lyuliyao@math.ucla.edu}

\author[4]{Siyu Guo}
\ead{guosiyu1@msu.edu}

\author[4,5]{Huan Lei}[orcid=0000-0002-9730-6481]
\cormark[1]
\ead{leihuan@msu.edu}


\affiliation[1]{organization={School of Artificial Intelligence and Data Science,
                              University of Science and Technology of China},
                city={Hefei}, state={Anhui}, country={China}}

\affiliation[2]{organization={Suzhou Institute for Advanced Research,
                              University of Science and Technology of China},
                city={Suzhou}, state={Jiangsu}, country={China}}

\affiliation[3]{organization={Suzhou Big Data \& AI Research and Engineering Center},
                city={Suzhou}, state={Jiangsu}, country={China}}

\affiliation[4]{organization={Department of Computational Mathematics, Science \& Engineering,
                              Michigan State University},
                city={East Lansing}, postcode={48824}, state={MI}, country={USA}}

\affiliation[5]{organization={Department of Statistics \& Probability,
                              Michigan State University},
                city={East Lansing}, postcode={48824}, state={MI}, country={USA}}

\begin{abstract}
Sampling from high-dimensional Gibbs measures poses a challenge when the energy landscape consists of multiple metastable states. 
Enhanced-sampling methods mitigate this difficulty by introducing adaptive biasing potentials to facilitate the exploration along prescribed collective variables (CVs), but their scalability is often limited by the dimension of the CV space. 
Motivated by the Wasserstein-gradient-flow interpretation of adaptive biasing, we propose a regularized path-dependent McKean--Vlasov formulation for high-dimensional enhanced sampling.   The formulation replaces the variational regularization of the Wasserstein functional by a direct regularization of the CV marginal density in the McKean--Vlasov drift, avoiding the outer convolution over the CV domain. Furthermore, 
it replaces the instantaneous law by a weighted path-history measure to improve statistical stability in the small-replica regime.
We establish well-posedness of the resulting regularized and path-dependent stochastic dynamics under suitable assumptions.  
For numerical realization, the history-averaged CV marginal density is approximated using an optimization-free functional hierarchical tensor representation, leading to a scalable density-based adaptive biasing scheme. 
Numerical experiments on benchmark potentials and molecular systems demonstrate the effectiveness of the proposed method for sampling problems with CV dimensions up to \(64\).

\end{abstract}

\maketitle

\section{Introduction}

Sampling from high-dimensional Gibbs measures \(
\pi({\rm d} \bx) \propto \exp\left(-\beta\,U(\bx)\right) {\rm d} \bx \) associated with complex energy landscapes $U(\bx)$ is a central problem in computational statistical mechanics. When $U(\bx)$ contains multiple well-separated local minima, the direct Langevin dynamics often exhibits rare barrier-crossing events and fails to explore configuration space within a feasible simulation time. Enhanced-sampling methods address this difficulty by introducing a biasing drift term, often in a prescribed space of collective variables (CVs), to promote exploration of metastable regions while permitting unbiased reconstruction of $\pi$-averages by reweighting. Such ideas have broad applications in molecular simulation \cite{shaw2010atomic}, Bayesian inference \cite{welling2011bayesian}, and the inverse problem \cite{stuart2010inverse}. However, the practical success of this strategy depends on the ability to construct and update the bias efficiently, and this task becomes increasingly difficult as the number of CVs grows. Developing computationally tractable adaptive-biasing formulations for moderate- to high-dimensional CV spaces remains an essential challenge.

Given a prescribed CV map \(\xi:\mathbb{R}^{d}\to\mathbb{R}^m\) with $m\ll d$, recent work \cite{lelievre2025convergence} by Leli\`evre, Lin, and Monmarch\'e provided a useful Wasserstein-gradient-flow perspective on adaptive biasing potential methods. 
In this formulation, the evolving probability law is induced by a modified free-energy functional with an additional entropy term for the CV marginal. The adaptive biasing drift takes a general form 
$\nabla\xi(\bx)^{\!\top}\nabla\log\rho_{\xi,t}(\xi(\bx))$ where \(\rho_{\xi,t}\)  denotes the instantaneous CV marginal density. Thus, the biased term 
penalizes the high-probability regions through the score function of the CVs. 
This perspective provides a unifying mathematical viewpoint for potential-of-mean-force (PMF)-based adaptive biasing methods and is closely related to classical enhanced-sampling strategies, including metadynamics ~\cite{laio2002escaping,barducci2008well}, umbrella sampling ~\cite{torrie1977nonphysical}, adaptive biasing force methods ~\cite{Darve_Pohorille_JCP_2001,darve2008adaptive} and their long-time convergence analysis ~\cite{lelievre2008long,chipot2011enhanced}, temperature-accelerated ~\cite{maragliano2006temperature,rosso2002use} and variationally enhanced sampling \cite{valsson2014variational,bonati2019neural}. Meanwhile, it also identifies the estimation and differentiation of an evolving CV marginal density as a central computational task in adaptive enhanced sampling.

Despite its conceptual appeal, the direct Wasserstein gradient flow formulation poses numerical challenges for high-dimensional CV spaces. 
The first challenge is that the biasing drift term depends on the score \(\nabla\log\rho_{\xi,t}\) of the evolving CV marginal density. 
In practice, the density estimates obtained from finite samples can exhibit large oscillations or loss of positivity in poorly sampled regions. Ref. ~\cite{lelievre2025convergence} proposed a regularized variational formulation to address the issue. The resulting gradient flow involves a nested convolution term
$\psi_\delta \ast \log(\psi_\delta \ast \rho_{\xi,t})$, 
where the inner convolution is taken against the empirical measure but the outer convolution acts over the CV domain. The latter one typically requires a grid representation over the full CV space and becomes computationally intractable in high-dimensional problems. 
A second difficulty is the finite-ensemble approximation error: the drift term depends on the instantaneous law of the process, which, in practice, is approximated from a finite number $M$ of interacting walkers on the fly. The mean-field description is justified only in the idealized limit $M\to \infty$, whereas realistic molecular simulations can typically only afford a small number (e.g., $\mathcal{O}(10)$) of walkers, so the instantaneous empirical marginal is subject to substantial fluctuations, and thus may not provide a reliable approximation of \(\rho_{\xi,t}\). 
These two limitations motivate a stochastic reformulation of the sampling process that enables more efficient construction of the drift term while remaining well-posed and numerically robust for finite-replica simulations.

In this study, we address the aforementioned difficulties by reformulating the stochastic dynamics of the adaptive biased sampling process. 
First, instead of regularizing the free-energy functional, we introduce a direct regularization of the CV marginal that defines the McKean--Vlasov drift. 
This construction yields a smooth and positivity-preserving CV marginal and therefore a well-defined score function in poorly sampled regions while avoiding the second convolution over the CV space. 
This modification sacrifices the exact Wasserstein gradient-flow structure, but leads to a drift term that enables substantially simpler numerical evaluation in the high-dimensional CV space. 
Second, to reduce the finite-ensemble approximation error due to the limited number of walkers, we replace the instantaneous law in the McKean--Vlasov drift by a weighted history measure accumulated along the trajectories. 
The reformulation leads to a path-distribution-dependent stochastic process ~\cite{wang2018distribution}  where the instantaneous biasing drift further depends on the history of the sampling process. 
Crucially, we show that, under suitable regularity assumptions, both the regularized McKean--Vlasov dynamics and its path-dependent generalization admit unique strong solutions. 
A related history-dependent adaptive biased force method has been analyzed by Bena\"im, Br\'ehier, and Monmarch\'e~\cite{benaim2020abf_self_interacting} on the well-posedness and long-time consistency, where the trajectory history is used to estimate the conditional mean force.
In contrast, the present formulation uses history measure to directly construct a regularized CV marginal density and the corresponding score-driven biasing drift.
%
With the new formulation, the remaining computational task reduces to the efficient approximation of the regularized CV marginal density in the high-dimensional CV space.

Several approaches have been explored for related high-dimensional bias or density approximation problems. Neural-network parametrizations \cite{galvelis2017neural,zhang2018reinforced,schneider2017stochastic,wang2022data} provide flexible representations of the bias, but typically require nonconvex training and may be difficult to control in sparsely sampled regions. Tensor-based methods offer a complementary approach by exploiting various low-rank structures such as Canonical Polyadic ~\cite{carroll1970analysis}, Tucker~\cite{tucker1966some}, tensor train ~\cite{oseledets2011tensor}, and hierarchical Tucker ~\cite{hackbusch2019hierarchical}, which may admit efficient storage and numerical operations linearly scaled with the dimension when such low-rank structure exists. Tensor-product and tensor-network representations have been used in probability density estimation \cite{Novikov_Panov_PMLR_2021, YoonHaeng_Jeremy_ACHA_2023,tang_generative_2023,tang2024solving},  high-dimensional partial differential equation solvers \cite{
Khoromskij_Schwab_SISC_2011,bachmayr_tensor_2016,Einkemmer_Joseph_JCP_2021, Tang_Collis_SISC_2025, wang2025tensor}, and rare-event simulation \cite{Chen_Hoskins_JCP_2023}. In this work, we use a functional hierarchical tensor (FHT) approximation \cite{tang2024solving} for the CV marginal density, where the density function is represented by a set of tensor-product bases and the coefficients are obtained in a low-rank hierarchical format using sketching. 
While Ref. \cite{tang2024solving} uses FHT density estimation to reconstruct solutions of high-dimensional Fokker--Planck equations from samples of the associated stochastic differential equations, 
the present work uses the FHT approximation in a feed-back loop where the accumulated trajectory data are used to estimate the CV marginal, construct the biasing drift, and update the path-dependent McKean--Vlasov dynamics to drive the exploration of the phase space.

The remainder of the paper develops this regularized path-dependent McKean--Vlasov framework and the computational method. 
We first explain why path-dependent dynamics are needed for adaptive biasing in the small-replica regime, providing the physical motivation in Section~\ref{subsec:prelim}. 
We then introduce the regularized and path-dependent formulations and establish their well-posedness in Sections~\ref{subsec:regularized} and~\ref{subsec:path}. 
The FHT-based density approximation and adaptive sampling algorithm are described in Section~\ref{subsec:density_estimation} with a summary of  the adaptive sampling algorithm in Section~\ref{subsec:adpative_sampling}. 
In Section~\ref{sec:result}, we present applications to molecular systems with CV dimensions up to \(64\), including both periodic and non-periodic CV settings. To the best of our knowledge, this CV dimensionality is beyond the regime typically addressed by adaptive enhanced sampling methods for molecular dynamics. 
Section~\ref{sec:discussion} concludes with a brief summary and discussion.

 \section{Method}\label{sec:method}

\subsection{Preliminaries and Wasserstein-inspired adaptive biasing}
\label{subsec:prelim}

We consider a molecular system with configuration variables $\bx \in \mathbb{R}^{d}$ evolving under a potential energy function $U : \mathbb{R}^{d} \to \mathbb{R}$. 
At inverse temperature \(\beta=(k_B T)^{-1}\), the target equilibrium distribution is the Gibbs measure
\begin{equation}
\pi(\rm d\bx)=Z^{-1}\exp\!\left(-\beta U(\bx)\right)\,\rm d\bx,
\label{eq:boltzmann}    
\end{equation}
where \(Z\) is the partition function. 
A standard stochastic dynamics preserving \(\pi\) as its invariant measure is the overdamped Langevin dynamics
\[
\rm d\bX_t
=
-\frac{1}{\gamma}\nabla U(\bX_t)\, \rm dt
+
\sqrt{\frac{2}{\beta\gamma}}\, \rm d\bW_t ,
\]
where \(\gamma>0\) is a friction coefficient and \(\bW_t\) is a standard \(d\)-dimensional Brownian motion. Although molecular dynamics (MD) simulations are often performed with underdamped Langevin dynamics, we formulate the stochastic analysis below in the overdamped setting for clarity. When \(U\) has multiple well-separated local minima, the dynamics is metastable and may fail to explore the full configuration space over feasible simulation times.

To overcome this challenge, a variety of enhanced-sampling methods have been developed. The standard approach is to introduce a smooth CV mapping
\[
    \xi : \mathbb{R}^{d} \to \mathbb{R}^m, \qquad m \ll d,
\]
designed to distinguish the system's metastable regions. Throughout this work, we assume that $\xi$ is a submersion, i.e., $\operatorname{rank} \nabla\xi(\bx) = m$ for all $\bx \in \mathbb{R}^{d}$.
While typical choices often rely on the physical insights (e.g.,  
backbone dihedral angles and interatomic distances), various data-driven learning of CVs becomes a promising yet open problem. In this study, we take $\xi$ as given and focus on the sampling problem in the prescribed CV space.

For a probability measure \(\mu\in\mathcal P_2(\mathbb R^d)\), we denote its pushforward under \(\xi\) by $\mu_\xi=\xi_\#\mu$. When $\mu$ and $\mu_\xi$ admit a Lebesgue density, we use $\rho_{\mu}(\bx)$ and $\rho_{\xi}(\bz)$ to denote the density functions. 
Under the submersion assumption, the coarea formula gives the CV marginal density in the form
\[
\rho_{\xi}(\bz)
=
\int_{\xi^{-1}(\bz)}
\frac{\rho_\mu(\bx)}{J_\xi(\bx)}\,d\sigma_{\bz}(\bx),
\qquad
J_\xi(\bx)
=
\sqrt{\det\!\bigl(\nabla\xi(\bx)\nabla\xi(\bx)^{\top}\bigr)},
\]
where \(d\sigma_{\bz}\) denotes the induced surface measure on the level set \(\xi^{-1}(\bz)\). 

Following~\cite{lelievre2025convergence}, many enhanced-sampling methods can be cast as Wasserstein gradient flows of a modified free-energy functional $\mathcal F_\alpha(\mu)$ by introducing an entropy penalty on the CV marginal, i.e., 
\begin{equation}
\mathcal F_\alpha(\mu)
=
\int_{\mathbb R^d} U(\bx)\,\mu(d\bx)
+
\frac{1}{\beta}
\int_{\mathbb R^d}
\rho_\mu(\bx)\log\rho_\mu(\bx)\, \intd\bx
+
\frac{\alpha}{\beta}
\int_{\mathbb R^m}
\rho_{\xi}(\bz)\log\rho_{\xi}(\bz)\,\intd\bz, 
\nonumber
\end{equation}
with \(\alpha\ge 0\) controls the bias strength. 
The first two terms represent the standard free energy whose unique minimizer is the Boltzmann--Gibbs distribution~\eqref{eq:boltzmann}.
The third CV-entropy term penalizes concentration of the law in CV space and corresponds to a biasing drift that drives the system toward a \emph{flatter} distribution in the CV space, effectively lowering the free-energy barriers separating metastable basins.
By Proposition~1 of~\cite{lelievre2025convergence}, the functional $\mathcal{F}_\alpha$ admits a \emph{unique minimizer}
\[
\rho_\alpha^*(\bx)
\propto
\exp\!\left[
-\beta\left(
U(\bx)-\frac{\alpha}{\alpha+1}A(\xi(\bx))
\right)
\right]
\]
where $A(\bz) =-\beta^{-1}\log \rho_{\pi,\xi}(\bz)$ is the free energy and $\rho_{\pi,\xi}$ is the marginal density of the equilibrium Gibbs distribution.

Formally, the Wasserstein gradient flow associated with \(\mathcal F_\alpha\) yields a nonlinear McKean--Vlasov dynamics. 
The corresponding biased overdamped dynamics takes the form
\begin{equation}
{\rm d}\bX_t
=
-\frac{1}{\gamma}
\left[
\nabla U(\bX_t)
+
\frac{\alpha}{\beta}
\nabla\xi(\bX_t)^{\top}
\nabla_{\bz}\log\rho_{\xi,t}\bigl(\xi(\bX_t)\bigr)
\right]{\rm d}t
+
\sqrt{\frac{2}{\beta\gamma}}\, {\rm d}\bW_t .
\label{eq:MV_ideal}
\end{equation}
The dynamics~\eqref{eq:MV_ideal} is only formal since the score \(\nabla_{\bz}\log\rho_{\xi,t}\) might fail to be smooth when the CV marginal is estimated from finite samples. 
One possible way to address this difficulty is to introduce a regularized free-energy functional.  
The corresponding regularized free-energy functional is
\begin{equation}\label{eq:Falpha_delta}
    \mathcal F_\alpha^\delta(\mu)
    =
    \int_{\mathbb R^d} U(\bx)\,\mu(d\bx)
    +
    \frac{1}{\beta}
    \int_{\mathbb R^d}
    \rho_\mu(\bx)\log\rho_\mu(\bx)\,d\bx
    +
    \frac{\alpha}{\beta}
    \int_{\mathbb R^m}
   (\varphi^\delta*\mu_\xi) \log (\varphi^\delta*\mu_\xi) \,d\bz, 
\end{equation}
where \(\varphi^\delta\) is a smooth mollifier. The corresponding regularized McKean--Vlasov SDE
\begin{equation}\label{eq:MV_delta}
    \intd \bX_t = -\frac{1}{\gamma}\nabla\!\left(U(\bX_t) + \frac{\alpha}{\beta}\,\bigl(\varphi^\delta * \log(\varphi^\delta * \mu_{\xi,t})\bigr)\!\bigl(\xi(\bX_t)\bigr) \right)\intd t + \sqrt{\frac{2}{\beta\gamma}}\,\intd \bW_t.
\end{equation}
The outer convolution regularizes the variational derivative in CV space and yields a smooth biasing drift under suitable assumptions on the mollifier. Here the inner convolution smooths the CV marginal, while the outer convolution regularizes the variational derivative in CV space and yields a smooth biasing drift under suitable assumptions on the mollifier and the SDE is well-posed.

Equation~\eqref{eq:MV_delta} can be approximated by a standard interacting particle system with \(M\) particles (i.e., replicas of MD simulations). 
However, two numerical challenges remain in high-dimensional CV spaces. 
First, the drift in~\eqref{eq:MV_delta} contains two nested convolutions. 
While the inner convolution can be evaluated directly against the empirical particles, the outer convolution acts on a general function over the CV domain and typically requires a grid- or basis-based representation over the full CV space, which becomes prohibitively expensive as \(m\) grows. 
Second, the drift depends on the instantaneous law \(\mu_t={\rm Law}(\bX_t)\), or equivalently on the instantaneous CV marginal density \(\rho_{\xi,t}\). 
The interacting-particle approximation approaches the mean-field limit only as \(M\to\infty\), whereas realistic MD simulations typically use only a small number of replicas, often on the order of \(\mathcal O(10)\). 
The resulting instantaneous empirical marginal is therefore too noisy to provide a reliable high-dimensional density estimate.

These two challenges motivate the reformulations introduced below. 
In Section~\ref{subsec:regularized}, we introduce a directly regularized drift that avoids the outer convolution while retaining a well-defined score field. 
In Section~\ref{subsec:path}, we replace the instantaneous law by a weighted history measure accumulated along the trajectories. 
Finally, Section~\ref{subsec:density_estimation} describes the FHT-based density approximation used to estimate the regularized CV marginal from accumulated trajectory data and Section. \ref{subsec:adpative_sampling} summarizes the sampling algorithm.

\subsection{Directly regularized McKean--Vlasov dynamics}
\label{subsec:regularized}

We now introduce a directly regularized McKean--Vlasov dynamics that avoids the outer convolution and meanwhile ensures the well-posedness. 
The key idea is to regularize the CV marginal density that enters the biasing drift rather than the free-energy functional itself. The regularization proceeds in two steps. First, we introduce a mollified CV marginal density $(\varphi^\delta * \mu_\xi)(\bz)$ that is everywhere smooth and strictly positive on the support of $\varphi^\delta * \mu_\xi$. However, the obtained density could exhibit pronounced oscillation and yield an unstable score function in the poorly sampled region. Therefore, we introduce a second regularization
\[
K_{\tau,\epsilon}(r)
=
\epsilon+\tau\,{\rm Softplus}(r/\tau),
\qquad
\epsilon,\tau>0,
\]
which ensures a uniform positive lower bound $K_{\epsilon,\tau}(\phi_\delta * \mu_\xi)$.  The regularized dynamics then reads
\begin{equation}\label{equ:regularized_MV_x}
    \intd \bX_t = -\frac{1}{\gamma}\nabla\!\left(U(\bX_t) + \frac{\alpha}{\beta}\log K_{\tau,\epsilon}\!\bigl((\varphi^\delta * \mu_{\xi,t})(\xi(\bX_t))\bigr)\right) \intd t + \sqrt{\frac{2}{\beta\gamma}}\,\intd \bW_t.
\end{equation}
Compared with the variationally regularized dynamics~\eqref{eq:MV_delta}, the drift in~\eqref{equ:regularized_MV_x} contains no outer convolution over the CV domain. Moreover, the mollification and the softplus regularization ensures that the drift term is globally Lipschitz (in the $W_2$ sense) which is discussed below in detail. 
On the other hand, we note that the dynamics \eqref{equ:regularized_MV_x} is no longer the exact Wasserstein gradient flow of the regularized free-energy functional~\eqref{eq:Falpha_delta}. 
Instead, it should be viewed as a directly regularized stochastic dynamics that trade the variational structure for a simple drift term without the nested convolution.

Although the gradient-flow structure is lost, the resulting dynamics remains well posed, as stated in the following theorem.

\begin{theorem}\label{thm:mk_well}
Assume that $\nabla U(\bx)$ is globally Lipschitz, that $\xi:\mathbb R^{d}\to\mathbb R^m$ belongs to $C^2$ with bounded $\nabla \xi$ and globally Lipschitz $\nabla^2 \xi$, and that $\varphi^\delta\in C^2(\mathbb R^m)$ with bounded derivatives. 
Then, for any deterministic initial condition \(\bX_0=\bx\), the McKean--Vlasov SDE~\eqref{equ:regularized_MV_x} admits a unique strong solution.
\end{theorem}

The proof relies on verifying the Lipschitz and linear-growth conditions required by the classical existence theory for McKean--Vlasov SDEs.

\begin{lemma} \label{lem:lip_cont}
Under the assumptions of Theorem~\ref{thm:mk_well},
let 
\begin{equation}
\bb(\bX,\mu)
=
-\frac{1}{\gamma}\nabla \left( U(\bX)
+
\frac{\alpha}{\beta}
\log\!
K_{\tau,\epsilon}\left((\varphi^\delta * \mu_{\xi})(\xi(\bX))\right)
\right), \quad \sigma = \sqrt{\frac{2}{\beta\gamma}}.
\label{eq:b_drift}
\end{equation}
Then there exists $C>0$ such that
\begin{equation}\label{equ:lip}
   \|\bb(\bX_1,\mu_1)-\bb(\bX_2,\mu_2)\|
\le
C\Big(\|\bX_1-\bX_2\|+W_2(\mu_1,\mu_2)\Big) 
\end{equation}
for any $\bX_1,\bX_2\in \mathbb R^{d}$ and $\mu_1,\mu_2\in \mathcal P_2(\mathbb{R}^{d})$.
As a consequence, $\|\bb(\bX,\mu)\| \leq C(1+\|\bX\|)$.
\end{lemma}

\begin{proof}
We decompose the drift into a potential and a bias component. Since $\nabla U$ is globally Lipschitz,
\[
\left\|-\frac1\gamma \nabla U(\bX_1)+\frac1\gamma \nabla U(\bX_2)\right\|
\le
\frac{\mathrm{Lip}(\nabla U)}{\gamma}\|\bX_1-\bX_2\|.
\]
It remains to estimate the bias term. Define
\[
f_\mu(\bz):=(\varphi^\delta * \mu_{\xi})(\bz)
=
\int_{\mathbb R^{d}} \varphi^\delta(\bz-\xi(\by))\,\mu(\intd \by),
\quad
H_\mu(\bz):=\nabla_\bz \log\!\big(K_{\tau,\epsilon}(f_\mu(\bz))\big).
\]
For any $\bz\in\mathbb R^m$ and any coupling $\pi\in\Pi(\mu_1,\mu_2)$,
\[
f_{\mu_1}(\bz)-f_{\mu_2}(\bz)
=
\iint
\Big(
\varphi^\delta(\bz-\xi(\by_1))
-
\varphi^\delta(\bz-\xi(\by_2))
\Big)\,\pi(\intd \by_1,\intd \by_2).
\]
Since $\nabla \varphi^\delta$ is bounded and $\xi$ is Lipschitz (because $\nabla \xi$ is bounded), taking the infimum over $\pi$ and supremum over $\bz$ yields
\[
\|f_{\mu_1}-f_{\mu_2}\|_{L^\infty(\mathbb R^m)}
\le
C\,W_1(\mu_1,\mu_2)
\le
C\,W_2(\mu_1,\mu_2).
\]
Similarly, since $\nabla \varphi^\delta$ is $C^1$ with bounded derivative,
$\|\nabla f_{\mu_1}-\nabla f_{\mu_2}\|_{L^\infty(\mathbb R^m)}
\le
C\,W_2(\mu_1,\mu_2)$.

Now set $h(r):=\log K_{\tau,\epsilon}(r)$.
Because $K_{\tau,\epsilon}\ge \epsilon>0$, both $h'$ and $h''$ are bounded. Therefore
\[
\|H_{\mu_1}-H_{\mu_2}\|_{L^\infty(\mathbb R^m)}
\le
C\Big(
\|f_{\mu_1}-f_{\mu_2}\|_{L^\infty}
+
\|\nabla f_{\mu_1}-\nabla f_{\mu_2}\|_{L^\infty}
\Big)
\le
C\,W_2(\mu_1,\mu_2).
\]

Next we control the dependence on $\bX$. Since $\varphi^\delta\in C^2$ with bounded derivatives, $H_\mu$ is globally Lipschitz in $\bz$, uniformly in $\mu$.
Using the boundedness of $\nabla \xi$ and the Lipschitz continuity of $\nabla \xi$ (which follows from the boundedness of $\nabla^2 \xi$), we obtain
\begin{align*}
&\|\nabla \xi(\bX_1)^\top H_{\mu_1}(\xi(\bX_1))
-
\nabla \xi(\bX_2)^\top H_{\mu_2}(\xi(\bX_2))\| \\
&\qquad \le
C\|\bX_1-\bX_2\|
+
C\|H_{\mu_1}(\xi(\bX_1))-H_{\mu_1}(\xi(\bX_2))\|
+
C\|H_{\mu_1}(\xi(\bX_2))-H_{\mu_2}(\xi(\bX_2))\| \\
&\qquad \le
C\|\bX_1-\bX_2\|+C\,W_2(\mu_1,\mu_2).
\end{align*}
Combining with the Lipschitz estimate for $\nabla U$ gives the result.
\end{proof}

\begin{proof}[Proof of Theorem~\ref{thm:mk_well}]
By Lemma~\ref{lem:lip_cont}, the assumptions of Theorem~1.7 in~\cite{carmona2016lectures} are satisfied, and dynamics~\eqref{equ:regularized_MV_x} admits a unique strong solution.
\end{proof}

\subsection{Path-dependent McKean--Vlasov equation}
\label{subsec:path}

The directly regularized McKean--Vlasov dynamics~\eqref{equ:regularized_MV_x} still depends on the instantaneous law 
\(\mu_t={\rm Law}(\bX_t)\). 
For numerical implementation, this law is approximated by an instantaneous empirical measure $\hat{\mu}_t^M = \frac{1}{M}\sum_{i=1}^M \delta_{\bX_t^i}$. 
However, realistic MD simulations can typically afford only a small number of replicas, so the instantaneous empirical measure $\hat{\mu}_t^M$ could be a poor estimator of $\mu_t$, especially after pushforward to a high-dimensional CV space.

To reduce this error, we replace the instantaneous law $\mu_t$ with a \emph{history measure} that accumulates the past of the process up to time $t$. This substitution trades an ensemble average (statistically noisy when $M$ is small) for a time average along the sample path, and the resulting dynamics can be implemented with only a modest number of replicas. Concretely, we define the path-history measure and its CV pushforward
\begin{equation}
\mu_t^q(\omega) = \int_0^1 \delta_{\bX_{ts}(\omega)}\, q(\intd s) \in \mathcal P_2(\mathbb R^{d}), \qquad \mu_{\xi,t}^q(\omega) = \xi_\#\mu_t^q(\omega) =\int_0^1 \delta_{\xi(\bX_{ts}(\omega))}\, q(\intd s), 
\label{eq:mu_q_path}
\end{equation}
where $q$ is a probability measure on $[0,1]$ that assigns weights to the historical trajectory.
For example, the uniform measure $q = \mathrm{Leb}|_{[0,1]}$ recovers the time-averaged empirical measure $\frac{1}{t}\int_0^t \delta_{\bX_r}\,\intd r$, while other choices of $q$ can place larger weight on more recent samples. 

The resulting dynamics is no longer a mean-field SDE. Substituting $\mu_{\xi,t}^q$ for $\mu_{\xi,t}$ in~\eqref{equ:regularized_MV_x} yields the path-dependent SDE
\begin{equation}\label{equ:path_dependent_sde}
\begin{aligned}
    \intd \bX_t &= -\frac{1}{\gamma}\nabla  \left( U(\bX_t) + \frac{\alpha}{\beta} \log K_{\tau,\epsilon} \left( (\varphi^\delta*\mu^q_{\xi,t} )(\xi(\bX_t))\right) \right) \intd t + \sqrt{\frac{2}{\beta\gamma}}\,\intd \bW_t\\
    &= \bb(\bX_t,\mu_t^q)\, \intd t +\sigma\, \intd \bW_t,
\end{aligned}
\end{equation}
where $\bb(\bX, \mu)$ is the drift term defined by Eq. \eqref{eq:b_drift}. Compared with ~\eqref{equ:regularized_MV_x}, the drift now depends on the full path $\{\bX_s\}_{0 \le s \le t}$ through the history measure $\mu_t^q$, and ~\eqref{equ:path_dependent_sde} becomes a \emph{path-distribution dependent} SDE.

\begin{theorem}\label{thm:path_wellposed}
Under the assumptions of Theorem~\ref{thm:mk_well}, for any 
\(q\in\mathcal P([0,1])\), the path-dependent SDE~\eqref{equ:path_dependent_sde}
admits a unique strong solution with initial condition \(\bX_0=\bx\).
\end{theorem}

The proof proceeds by Picard iteration. 
At each step \(k\), the history measure \(\mu_s^{k,q}\) generated by the previous iterate is frozen, reducing the equation to a standard SDE with Lipschitz coefficients whose well-posedness follows from classical theory~\cite{prevot2007concise}. 
The key observation is that, for two continuous paths \(\bX\) and \(\bY\),
\[
W_2\!\left(\mu_t^q[\bX],\mu_t^q[\bY]\right)^2
\le
\int_0^1 \|\bX_{ts}-\bY_{ts}\|^2\,q(ds)
\le
\sup_{0\le r\le t}\|\bX_r-\bY_r\|^2 .
\]
With the Lipschitz estimate for \(\bb\), this yields convergence of the Picard iterates \(\bX_t^k\) in 
\(L^2(\Omega;C([0,T];\mathbb R^d))\). 
We have the estimate
\[
\mathbb E[\Delta_T^{k+1}]
\le
\frac{C_T^k}{k!}\,\mathbb E[\Delta_T^1],
\qquad
\Delta_t^k
=
\sup_{s\le t}\|\bX_s^k-\bX_s^{k-1}\|^2 .
\]
Uniqueness follows from the same Lipschitz estimate and Gronwall's inequality applied to any two solutions. 
The complete proof is given in Appendix~\ref{app:proof_wellposed}.

The well-posedness result above does not by itself characterize the long-time behavior of the path-dependent dynamics. 
We next show that, under strong dissipativity assumptions, history-dependent McKean--Vlasov dynamics can be asymptotically consistent with the invariant measure of the corresponding instantaneous-law dynamics.
Recall that the long-time behavior of McKean--Vlasov SDE such as 
\eqref{equ:regularized_MV_x} has been studied extensively in recent years; see,
for instance, Theorem~3.1 in~\cite{wang2018distribution}. Let us consider 
\[
    {\rm d}\bX_t
    =
    \bb\bigl(\bX_t,{\rm Law}(\bX_t)\bigr)\,{\rm d}t
    +
    \sigma\,{\rm d}\bW_t ,
\]
and let $P_t^*$ denote its nonlinear semigroup. Assume that $\bb$ is continuous
on $\mathbb R^{d}\times\mathcal P_2$ and bounded on bounded sets. Moreover,
assume that there exist constants $\lambda_0>\lambda_1\ge 0$, $a_0>0$, and
$K_0\ge 0$ such that, for all
$\bx,\by\in\mathbb R^{d}$ and $\mu,\nu\in\mathcal P_2$,
\[
2\langle \bb(\bx,\mu)-\bb(\by,\nu),\bx-\by\rangle
\le
-\lambda_0|\bx-\by|^2
+
\lambda_1 W_2(\mu,\nu)^2,
\]
and
\[
2\langle \bb(\bx,\mu),\bx\rangle
\le
-a_0|\bx|^2
+
K_0\bigl(1+\mu(|\cdot|^2)\bigr).
\]
Then $P_t^*$ admits a unique invariant probability measure
$\mu_{\rm MV}^*\in\mathcal P_2$. Moreover, for every initial law
$\nu\in\mathcal P_2$,
\[
    W_2(P_t^*\nu,\mu_{\rm MV}^*)^2
    \le
    C(\nu)e^{-ct}
\]
for some constants $C(\nu)>0$ and $c>0$.

The dissipativity condition above is restrictive and is not expected to hold for arbitrary molecular potentials. However, under this assumption, replacing the
instantaneous law by a weighted-history dynamics are asymptotically consistent with the invariant measure 
of the corresponding instantaneous-law McKean--Vlasov dynamics in a precise Wasserstein sense.
For a time-weight family $q=(q_t)_{t\ge0}$ with
$q_t\in\mathcal P([0,1])$, we define
\[
    \mu_t^q[\bX]
    :=
    \int_0^1 \delta_{\bX_{ts}}\,q_t({\rm d}s),
    \qquad
    \ell_t^q[\bX]
    :=
    \int_0^1 {\rm Law}(\bX_{ts})\,q_t({\rm d}s).
\]
for a continuous process $\bX_t$.
The first object is the path-history measure of the trajectory, while the
second is the corresponding weighted history of past laws. The fixed-weight history used in~\eqref{equ:path_dependent_sde} is recovered by taking
$q_t\equiv q$.
For $\varepsilon\in(0,1]$, let $\mathcal Q_1(\varepsilon)$ and
$\mathcal Q_2(\varepsilon)$ denote the admissible classes of time-weight families
called $\Pi_1(\varepsilon)$ and $\Pi_2(\varepsilon)$ in~\cite{du2023empirical}.

\begin{proposition}
\label{prop:history_consistency}
Assume that the drift $\bb$ satisfies the monotonicity and moment assumptions
above, together with the $(2+\varrho)$-moment condition required in
\cite{du2023empirical} for some $\varrho>0$. We set
$\kappa=\frac{\varrho}{(d+2)(\varrho+2)}.
$
First, let $q\in\mathcal Q_1(\varepsilon_1)$ and
$\bar q\in\mathcal Q_2(\varepsilon_2)$, and consider the law-history dynamics
\[
    \intd \bX_t
    =
    \bb\bigl(\bX_t,\ell_t^q[\bX]\bigr)\,{\rm d}t
    +
    \sigma\,{\rm d}\bW_t .
\]
Then, for every initial law $\nu\in\mathcal P_{2+\varrho}$, there exists a
constant $C>0$ such that
\[
    \mathbb E_\nu
    \left[
        W_2\!\left(\mu_t^{\bar q}[\bX],\mu_{\rm MV}^*\right)^2
    \right]
    \le
    C t^{-\varepsilon},
    \qquad
    \varepsilon
    =
    \varepsilon_1\wedge \kappa\varepsilon_2 .
\]

Second, let
$q\in\mathcal Q_1(\varepsilon_1)\cap\mathcal Q_2(\varepsilon_2)$ and
$\bar q\in\mathcal Q_2(\varepsilon_2)$, and consider the self-interacting
history-dependent dynamics
\[
    {\rm d}\bX_t
    =
    \bb\bigl(\bX_t,\mu_t^q[\bX]\bigr)\,{\rm d}t
    +
    \sigma\,{\rm d}\bW_t .
\]
Then, for every deterministic initial condition \(\bX_0=\bx\), there exists a
constant $C>0$ such that
\[
    \mathbb E_{\bx}
    \left[
        W_2\!\left(\mu_t^{\bar q}[\bX],\mu_{\rm MV}^*\right)^2
    \right]
    \le
    C t^{-\varepsilon}.
\]
Consequently, both the law-history dynamics and the self-interacting
empirical-history dynamics are asymptotically consistent with the invariant
measure $\mu_{\rm MV}^*$ of the corresponding instantaneous-law McKean--Vlasov
dynamics.
\end{proposition}

\begin{proof}
This is a direct consequence of Theorems~2.1 and~2.2 of
\cite{du2023empirical}, rewritten in the notation used here and specialized to
the constant-diffusion setting.
\end{proof}

\subsection{Density estimation via functional hierarchical tensor}
\label{subsec:density_estimation}


The path-dependent formulation replaces the instantaneous CV marginal by a history-averaged measure that enables a small number of MD replicas. The numerical implementation still requires an efficient representation of the corresponding CV marginal density 
$\rho_t(\bz) \approx
(\phi_\delta * \widehat\mu_{\xi, n}^{M,q})(\bz)$, where $\widehat\mu_{\xi, t}^{M,q}$ is the finite-replica history measure of $\mu_{\xi, t}^q$ defined in Eq. \eqref{eq:mu_q_path}. In moderate- to high-dimensional CV spaces ($m\sim 10$--$100$), grid-based representations of \(\rho_t\) are infeasible.  We therefore approximate \(\rho_t\) using a functional hierarchical tensor (FHT)  approach \cite{tang2024solving} fitted from the accumulated CV samples. The key idea is to represent the CV marginal in a smooth basis and then compress the resulting coefficient tensor in a hierarchical low-rank format. This yields an approximation that is both smooth and computationally tractable.

Let $\bz=(z_1,\dots,z_m)\in\Omega\subset\mathbb{R}^m$ denote the CVs. For simplicity, we assume that each coordinate has been scaled to the reference interval $[-1,1]$; otherwise, an affine rescaling can be applied. On each coordinate $z_k\in[-1,1]$, we introduce $p$ Gaussian basis functions
\begin{equation}
    \psi_{j,\delta}^{(k)}(z_k)
    = \exp\!\left(-\frac{(z_k-c_j)^2}{2\delta^2}\right),
    \qquad
    c_j=-1+\frac{2(j-1)}{p-1},
    \qquad
    j=1,\dots,p,
    \label{eq:gaussian_basis}
\end{equation}
with equally spaced centers $\{c_j\}_{j=1}^p$ and a common width parameter $\delta$. Since all coordinates share the same reference domain, the same centers and width are used for every $k$.
Let \(\{\widetilde\psi_{j,\ell}^{(k)}\}_{j=1}^p\) denote the orthonormalized basis. The density is approximated by the tensor-product expansion
\begin{equation}
\rho_t(\bz)
\approx
\rho_{\rm FHT}^{(t)}(\bz)
=
\sum_{j_1,\ldots,j_m=1}^p
C^{(t)}_{j_1,\ldots,j_m}
\prod_{k=1}^m
\widetilde\psi_{j_k,\ell}^{(k)}(z_k).
    \label{eq:tensor_expansion}
\end{equation}
where $C\in\mathbb{R}^{p\times\cdots\times p}$ is the order-$m$ coefficient tensor.

A direct representation of $C$ requires $p^m$ coefficients and is therefore infeasible when $m$ is moderate or large.
Instead, FHT approximation replaces this full tensor by a hierarchical low-rank factorization in form of a binary tree. 
For an internal node \(\mathsf n\) of the binary tree, let \(a=a(\mathsf n)\) and \(b=b(\mathsf n)\) be its two child index sets, and let
\[
r=r(\mathsf n):=[m]\setminus(a\cup b)
\]
denote the remaining coordinates. 
The hierarchical low-rank structure assumes that, locally,
\begin{equation}
C^{(t)}(i_1,\ldots,i_m)
=
\sum_{\alpha,\beta,\theta}
C_a(i_a,\alpha)\,
C_b(i_b,\beta)\,
G_{\mathsf n}(\alpha,\beta,\theta)\,
C_r(\theta,i_r),
\label{eq:local_factorization}
\end{equation}
where \(i_a,i_b,i_r\) denote grouped basis indices, and \(G_{\mathsf n}\) is the local core tensor at node \(\mathsf n\). 
When the hierarchical ranks are bounded by \(R\), the number of parameters and the cost of evaluation scale linearly with the CV dimension \(m\), up to rank-dependent factors.

The local cores are estimated from accumulated CV samples
$
\mathcal D_t=\{\bz_\ell\}_{\ell=1}^{N_s},
\bz_\ell=\xi(\bX_{t_\ell}^{(i_\ell)}).
$
Following the sketching construction of~\cite{tang2024solving},
$S_a$, $S_b$, and $S_f$ be sketch tensors for the three coordinate groups, and let
$s_a$, $s_b$, and $s_f$ be the associated sketch functions built from the orthonormalized Gaussian basis.
Contracting~\eqref{eq:local_factorization} with the sketch tensors yields the sketched linear system
\begin{equation}
    B_q(r,s,t)
    =
    \sum_{\alpha,\beta,\theta}
    A_a(r,\alpha)\,
    A_b(s,\beta)\,
    A_f(t,\theta)\,
    G_{\mathsf n}(\alpha,\beta,\theta),
    \label{eq:sketched_tensor_system}
\end{equation}
or, equivalently,
\begin{equation}
    (A_a\otimes A_b\otimes A_f)\,G_{\mathsf n} = B_{\mathsf n}.
    \label{eq:sketched_system}
\end{equation}
The matrices $A_a$, $A_b$, and $A_f$ are the sketch-compressed representations of the block factors. The right-hand side $B_q$ is estimated directly from the samples by
\begin{equation}
    B_{\mathsf n}(r,s,t)
    \approx
    \frac{1}{N_s}\sum_{\ell=1}^{N_s}
    s_a(\bz^{(\ell)}_a,r)\,
    s_b(\bz^{(\ell)}_b,s)\,
    s_f(\bz^{(\ell)}_f,t).
    \label{eq:Bq_empirical}
\end{equation}

The matrices $A_a$, $A_b$, and $A_f$ are obtained from auxiliary sketched moment matrices. For example,
\begin{equation}
    Z_{a;\bar a}(r,u)
    \approx
    \frac{1}{N_s}\sum_{\ell=1}^{N_s}
    s_a(\bz^{(\ell)}_a,r)\,
    s_{\bar a}(\bz^{(\ell)}_{\bar a},u),
    \label{eq:Zaabar_empirical}
\end{equation}
and a rank-$r_a$ factorization of $Z_{a;\bar a}$ provides $A_a$.
Similarly, $A_b$ and $A_f$ are obtained from $Z_{b;\bar b}$ and $Z_{\bar f;f}$, respectively.
Once these reduced matrices are available, the local core tensor $G_{\mathsf n}$ is obtained by solving~\eqref{eq:sketched_system}.
\begin{remark}[Periodic treatment of torsional collective variables]
Equations~(10)--(11) are written for non-periodic CV coordinates on an
interval. For a torsional CV, however, the physical CV space is the torus \(\mathbb T=\mathbb R/(2\pi\mathbb Z)\), and the interval
\([-\pi,\pi)\) is only a coordinate representation. In all molecular examples
involving dihedral angles, the Gaussian features are therefore evaluated with
periodic boundary conditions in the CV space.
More precisely, for an angular coordinate \(\theta\in[-\pi,\pi)\), we use
periodic centers
\[
    c_j=-\pi+\frac{2\pi j}{p},\qquad j=0,\ldots,p-1,
\]
where the endpoints \(-\pi\) and \(\pi\) are not treated as distinct centers.
The non-periodic Gaussian basis in~(10) is replaced by the periodized Gaussian
basis
\[
    \psi^{\rm per}_{j,\delta}(\theta)
    =
    \sum_{\ell\in\mathbb Z}
    \exp\left(
        -\frac{(\theta-c_j+2\pi \ell)^2}{2\delta^2}
    \right).
\]
Equivalently, this basis assigns nonzero weight across the artificial cut at
\(-\pi/\pi\): samples near \(\pi\) contribute to basis functions centered near
\(-\pi\), and vice versa.
\end{remark}

\subsection{Adaptive sampling algorithm} \label{subsec:adpative_sampling}

With the FHT approximation \(\rho_{\rm FHT}^{(t)}\), we apply the  positive regularization 
$
\widetilde\rho_t(\bz)
=
K_{\tau,\epsilon}\!(\rho_{\rm FHT}^{(t)}(\bz))
$ and the bias potential used in the next sampling stage is
$
V_{\rm bias}^{(t)}(\bx)
=
\frac{\alpha}{\beta}\log \widetilde\rho_t(\xi(\bx)).
$
The practical implementation follows an outer-loop adaptive procedure.
At iteration time step $n$, we first generate new trajectory data using the current biased dynamics,
then update the accumulated CV dataset, fit a smooth approximation of the CV marginal
density by the FHT sketching procedure, regularize the fitted density to ensure positivity
and stability, and finally define a new bias potential from the logarithm of the regularized
density. 
After the adaptive stage, the final bias \(V_{\rm bias}^{(T_{max})}\) is frozen and production trajectories are generated under the biased potential
$
U(\bx)+V_{\rm bias}^{(T_{max})}(\xi(\bx)).
$
For samples from this fixed biased ensemble, unbiased averages with respect to the original Gibbs measure are recovered using weights
\[
w(\bx)
\propto
\exp\!\left(\beta V_{\rm bias}^{(T_{max})}(\xi(\bx))\right).
\]
Algorithm 1 summarizes the present path-dependent McKean-Vlasov-based sampling method. 


\begin{algorithm}[ht]
\caption{Path-dependent McKean-Vlasov adaptive enhanced sampling using FHT density estimation.}
\label{alg:adaptive}
\begin{algorithmic}[1]
\REQUIRE Initial configurations $\{\bX_0^{(i)}\}_{i=1}^M$; bias strength $\alpha$; FHT parameters $(p,\delta,r)$ and binary dimension tree $\mathcal{T}$; Softplus parameters $(\epsilon,\tau)$; number of bias updates $T_{\max}$; MD steps per update $N_{\mathrm{step}}$; integration step $\Delta t$; friction $\gamma$; inverse temperature $\beta$
\ENSURE Converged bias potential $V_{\mathrm{bias}}^{(T_{\max})}$; reweighted free-energy surface $F$
 
\STATE \textbf{Initialization:} $\mathcal{D}^{(0)} \leftarrow \emptyset$; \ $\widetilde{\rho}^{(0)} \leftarrow \text{const}$; \ $V_{\mathrm{bias}}^{(0)} \leftarrow 0$ \COMMENT{no initial bias}
 
\FOR{$t = 1, 2, \dots, T_{\max}$}
 
    \STATE \emph{\# Step A --- Trajectory collection under the current bias}
    \STATE Form the biased force field $\bF^{(t-1)}(\bx) = -\nabla U(\bx) - \nabla\xi(\bx)^\top \nabla_{\boldsymbol{\xi}} V_{\mathrm{bias}}^{(t-1)}\!\bigl(\xi(\bx)\bigr)$  
    \FOR{each walker $i = 1, \dots, M$ \textbf{in parallel}}
        \STATE Propagate $N_{\mathrm{step}}$ MD steps from $\bX^{(i)}$ under $\bF^{(t-1)}$ at temperature $T$ using the chosen integrator 
        \STATE Record CV snapshots $\{\xi(\bX_n^{(i)})\}$ every $N_{\mathrm{save}}$ steps and append to $\mathcal{D}_{\mathrm{new}}$
    \ENDFOR
    \STATE $\mathcal{D}^{(t)} \leftarrow \mathcal{D}^{(t-1)} \cup \mathcal{D}_{\mathrm{new}}$
 
    \STATE \emph{\# Step B --- Coordinate rescaling to $[-1,1]^m$}
    \STATE For each coordinate $k=1,\dots,m$, compute $[\xi_k^{\min},\xi_k^{\max}]$ from $\mathcal{D}^{(t)}$ and apply
    \STATE \quad $\widehat{\xi}_k = 2\,\dfrac{\xi_k - \xi_k^{\min}}{\xi_k^{\max}-\xi_k^{\min}} - 1$, \quad with Jacobian factor $J_k = 2/(\xi_k^{\max}-\xi_k^{\min})$
 
    \STATE \emph{\# Step C --- FHT density estimation via sketching}
    \STATE Build the Gaussian basis $\{\psi_{j,\delta}^{(k)}\}_{j=1}^{p}$ of \eqref{eq:gaussian_basis} and orthonormalize it to $\{\widetilde\psi_{j,\delta}^{(k)}\}_{j=1}^p$
    \STATE Draw the sketch tensors $\{S_a,S_b,S_f\}$ and assemble the sketch functions $\{s_a,s_b,s_f\}$
    \FOR{each internal node $\mathsf n \in \mathcal{T}$ in bottom-up order}
        \STATE Let $a = a(\mathsf n)$, $b = b(\mathsf n)$, $f = [m]\setminus(a\cup b)$
        \STATE Estimate the moment matrices $Z_{a;\bar a}$, $Z_{b;\bar b}$, $Z_{\bar f;f}$ from samples via~\eqref{eq:Zaabar_empirical}
        \STATE Compute truncated rank-$r$ factorizations to obtain the block factors $A_a$, $A_b$, $A_f$
        \STATE Estimate the sketched right-hand side $B_{\mathsf n}$ from samples via~\eqref{eq:Bq_empirical}
        \STATE Solve $(A_a \otimes A_b \otimes A_f)\, G_\mathsf n = B_{\mathsf n}$ for the local core $G_{\mathsf n}$
    \ENDFOR
    \STATE Assemble the coefficient tensor from $\{G_{\mathsf n}\}$ and $\{A_a,A_b,A_f\}$ along $\mathcal{T}$ via \eqref{eq:local_factorization}; obtain $\widehat{\rho}_{\mathrm{FHT}}^{(t)}(\widehat{\boldsymbol{\xi}})$ through the expansion \eqref{eq:tensor_expansion}
    \STATE Map back to the original coordinates: $\rho_{\mathrm{FHT}}^{(t)}(\boldsymbol{\xi}) = \widehat{\rho}_{\mathrm{FHT}}^{(t)}\bigl(\widehat{\boldsymbol{\xi}}(\boldsymbol{\xi})\bigr)\,\prod_{k=1}^{m} J_k$
 
    \STATE \emph{\# Step D --- Softplus regularization and bias update}
    \STATE $\widetilde{\rho}^{(t)}(\boldsymbol{\xi}) \leftarrow K_{\tau,\epsilon}\!\bigl(\rho_{\mathrm{FHT}}^{(t)}(\boldsymbol{\xi})\bigr)$, \ where $K_{\tau,\epsilon}(x) = \epsilon + \tau\,\mathrm{Softplus}(x/\tau)$
    \STATE $V_{\mathrm{bias}}^{(t)}(\boldsymbol{\xi}) \leftarrow \alpha\, k_B T \log \widetilde{\rho}^{(t)}(\boldsymbol{\xi})$
\ENDFOR
 
\STATE \emph{\# Production run and FES reconstruction}
\STATE Fix the bias at $V_{\mathrm{bias}}^{(T_{\max})}$ and run $M'$ independent production trajectories under the biased force field $\bF^{(T_{\max})}$
\STATE Assign to each production sample $\boldsymbol{\xi}$ the reweighting weight $w(\boldsymbol{\xi}) \propto \exp\!\bigl(\beta\, V_{\mathrm{bias}}^{(T_{\max})}(\boldsymbol{\xi})\bigr)$
\STATE Reconstruct the unbiased FES $F(\boldsymbol{\xi}) = -k_B T \log \rho_\xi(\boldsymbol{\xi})$ using fixed-bias reweighting, or via MBAR~\cite{shirts2008statistically} when production samples are pooled across multiple bias stages
\end{algorithmic}
\end{algorithm}

\section{Numerical Results}\label{sec:result}
We assess the proposed method on benchmark systems of increasing complexity, ranging from the analytical M\"uller potential to solvated peptides and a protein benchmark. In each case we examine two complementary aspects: the efficiency of barrier crossing during the adaptive stage, and the quality of the reweighted free-energy surface (FES) recovered by reweighting after the bias is frozen. Throughout this section, the FES is defined up to an additive constant by
\begin{equation}
    F(\boldsymbol{\xi}) = -k_B T \log \rho_{\xi}(\boldsymbol{\xi}) + C,
\end{equation}
where $\rho_{\xi}$ denotes the pushforward density in the CV space. When one- or two-dimensional marginals are reported, the remaining CVs are integrated out before applying the logarithmic transformation.

\subsection{M\"uller potential}
\label{subsec:muller}

We first consider the M\"uller potential~\cite{muller1979location}, a two-dimensional model landscape that provides a controlled test for both enhanced exploration and FES reconstruction. The potential is defined by
\begin{equation}
    U(x, y) = \sum_{i=1}^{4} A_i \exp\!\Big[ a_i (x - x_i^0)^2 + b_i (x - x_i^0)(y - y_i^0) + c_i (y - y_i^0)^2 \Big],
    \label{eq:muller}
\end{equation}
where the parameters $\{A_i, a_i, b_i, c_i, x_i^0, y_i^0\}$ are given in~\cite{muller1979location}. The landscape contains three metastable wells separated by pronounced barriers, making it a stringent benchmark for barrier-crossing dynamics. Since the physical coordinates already form a two-dimensional CV space, we work directly with $\boldsymbol{\xi}=(x,y)$ and simulate the biased overdamped Langevin dynamics with time step $\Delta t = 0.005$, friction coefficient $\gamma = 5.0$, and thermal energy $k_B T = 2.5$.

The bias is updated every $100{,}000$ time steps, and all samples collected so far are used to refit the density estimator. The FHT approximation employs $p = 31$ Gaussian basis functions per coordinate, width $\sigma = 0.2$, and maximal bond dimension $r = 15$. The Softplus regularization parameters are set to $\epsilon = 0.1$ and $\tau = 0.1$.

To quantify the influence of the bias strength, we test $\alpha = 13$, $16$, and $20$. Figure~\ref{fig:muller_combined} shows that the unbiased dynamics remains trapped in a single well and exhibits no inter-basin transition over the simulated time horizon. In contrast, the biased dynamics generates frequent transitions among all three basins, with the number of observed crossings increasing from $26$ to $120$ and then to $196$ as $\alpha$ is increased. This monotone trend is consistent with the effective-temperature interpretation derived in Section~2.2: a larger $\alpha$ produces a flatter CV-space distribution and facilitates barrier crossing.

After $150$ bias updates, the bias potential is frozen and a production trajectory of $6 \times 10^6$ steps is generated for reweighting. The reweighted FES in Figure~\ref{fig:muller_combined}(f) accurately reproduces the three dominant metastable basins and the low-free-energy channels connecting them. The difference map in Figure~\ref{fig:muller_combined}(g) shows that the remaining discrepancies are concentrated near the boundary of the sampled region and in sparsely visited high-energy areas, while the low-energy portion of the landscape is recovered with good fidelity.

\begin{figure}
    \centering
    \includegraphics[width=0.95\linewidth]{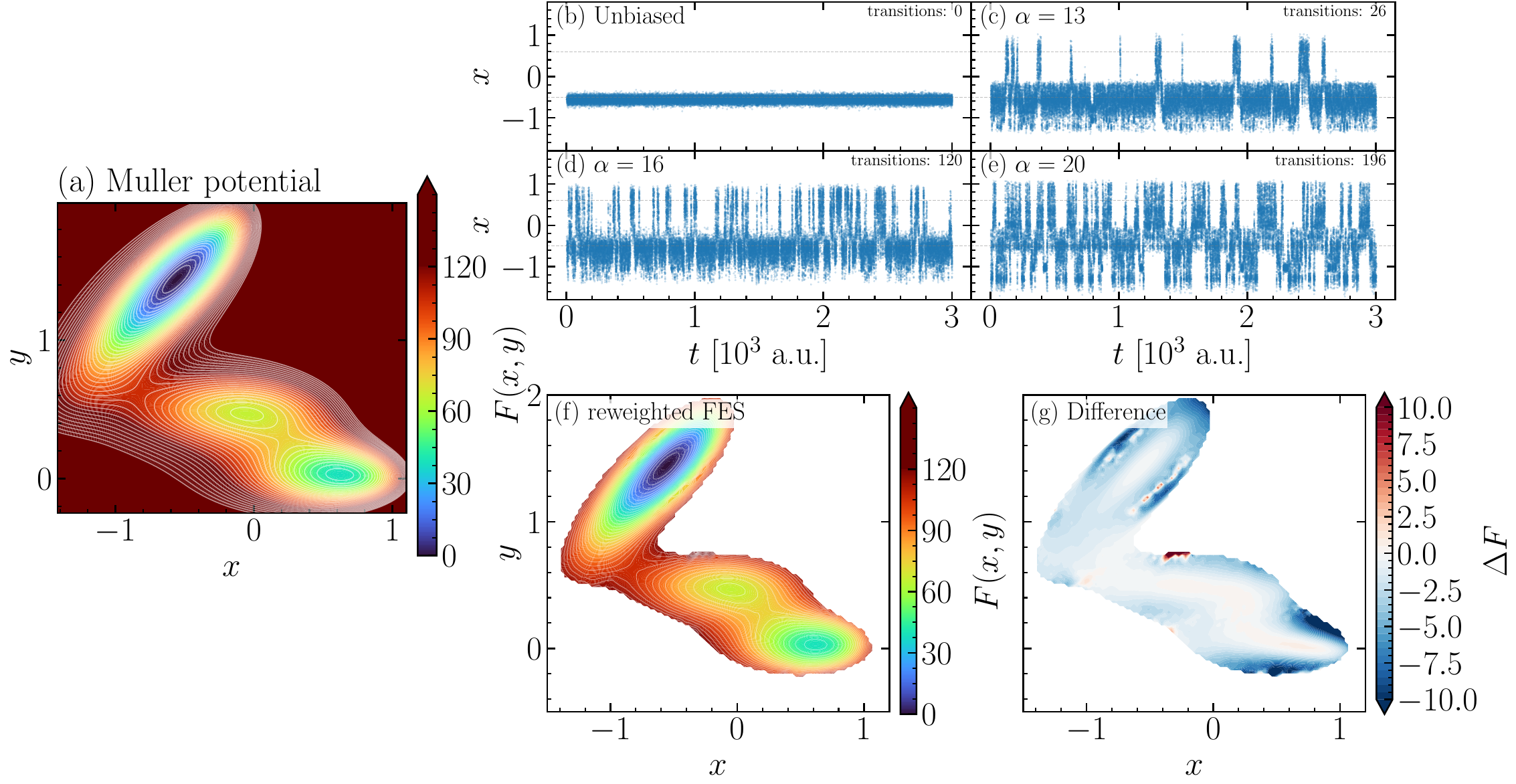}
    \caption{Benchmark of the 2D M\"uller potential. (a) Contours of the reference potential~\eqref{eq:muller}. (b)--(e) Time series of the $x$ coordinate for unbiased dynamics and for biased dynamics with $\alpha = 13$, $16$, and $20$; the numbers in the panel titles report the corresponding counts of inter-basin transitions over the same simulation horizon. (f) Reweighted FES obtained after freezing the bias. (g) Pointwise difference between the reweighted FES and the reference surface on the sampled domain.}
    \label{fig:muller_combined}
\end{figure}

\subsection{Alanine systems}
Alanine peptides provide standard molecular benchmarks with increasingly rich torsional free-energy landscapes. We first study alanine dipeptide in a two-dimensional CV space, and then consider the four-dimensional ACE-(ALA)$_2$-NME system.

\subsubsection{ACE-ALA-NME system}
We begin with alanine dipeptide (Ace--Ala--Nme, denoted Ala2), a canonical test case for conformational sampling. The solute is solvated in $383$ TIP3P water molecules, and all-atom MD simulations are performed in the canonical (NVT) ensemble at $T = 300$ K using the Amber99-SB force field~\cite{hornak2006comparison}. The time step is $0.001$ ps. Periodic boundary conditions are applied in all spatial directions. Van der Waals interactions are truncated at $0.9$ nm, and long-range electrostatics are treated by the smooth particle mesh Ewald (PME) method with real-space cutoff $0.9$ nm and reciprocal-space grid spacing $0.12$ nm.

The two backbone torsion angles $\phi(C,N,C_{\alpha},C)$ and $\psi(N,C_{\alpha},C,N)$ are used as CVs. The bias is updated every $200{,}000$ MD steps, and configurations are recorded every $50$ steps for density estimation and post-processing. Figure~\ref{fig:ala2_exploration} shows that the unbiased trajectory remains concentrated in a restricted angular region, whereas the biased trajectories progressively cover the full $\phi$ range as $\alpha$ increases from $0.5$ to $2.0$. The corresponding empirical marginal density becomes markedly flatter, indicating that the adaptive bias suppresses repeated visits to the most frequently sampled torsional states.

After $20$ bias updates, the bias is frozen and $32$ independent production simulations are performed for FES estimation. Figure~\ref{fig:ala2_fes_convergence} reports the reweighted one-dimensional marginals $F(\phi)$ and $F(\psi)$ at $t=500$, $1000$, $2000$, and $4000$ ps. Relative to unbiased MD, the biased dynamics substantially reduces the run-to-run variability and recovers the locations of the principal wells and barriers at much earlier times. The improvement is systematic with respect to $\alpha$: even $\alpha=0.5$ accelerates convergence, while $\alpha=1.0$ and $2.0$ yield the most stable profiles in the early part of the production stage.

\begin{figure}
    \centering
    \includegraphics[width=0.95\linewidth]{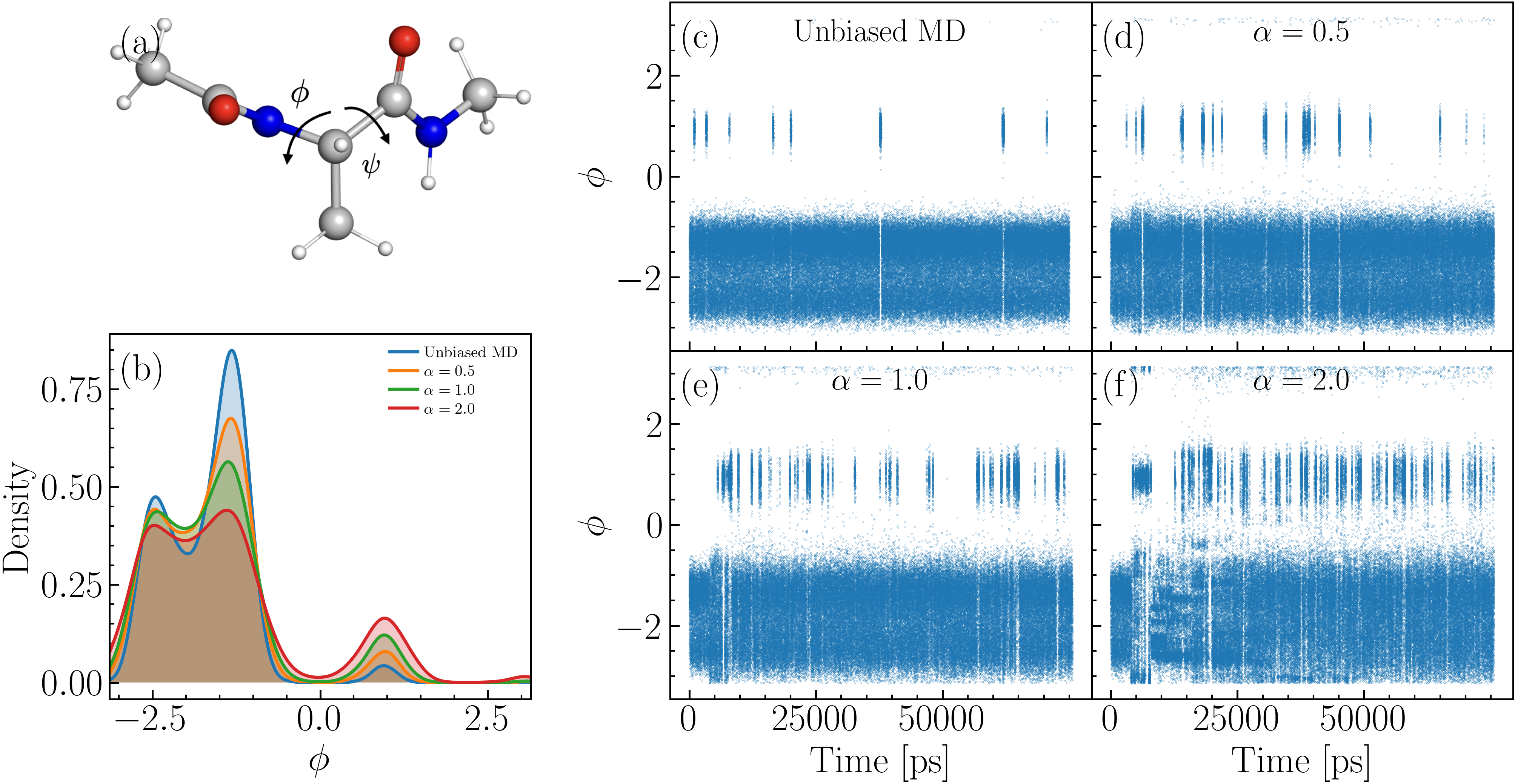}
    \caption{Benchmark of the Ala2 molecule with 2 CVs. (a) Molecular structure of alanine dipeptide. (b) Estimated marginal density of $\phi$ obtained from trajectories generated with different bias strengths. (c)--(f) Time series of the dihedral angle $\phi$ for unbiased MD and for biased dynamics with $\alpha=0.5$, $1.0$, and $2.0$, respectively. Increasing $\alpha$ broadens the sampled angular range and produces more frequent transitions between metastable regions.}
    \label{fig:ala2_exploration}
\end{figure}

\begin{figure}
    \centering
    \includegraphics[width=0.95\linewidth]{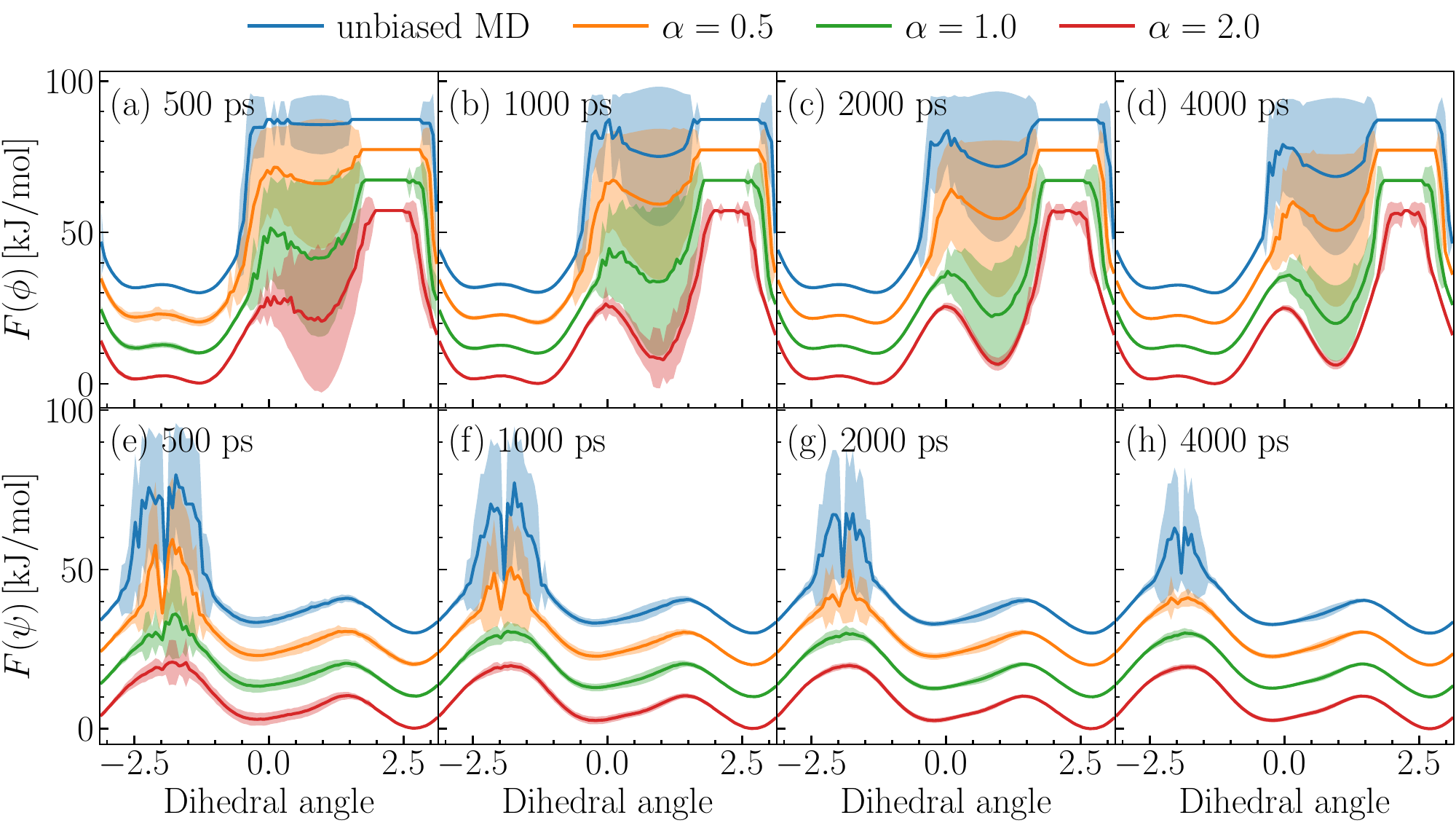}
    \caption{Convergence of the marginal free-energy surfaces $F(\phi)$ (top row) and $F(\psi)$ (bottom row) for Ala2 after the bias is frozen at iteration $20$. Results are averaged over $32$ independent production runs. Columns correspond to production lengths $t=500$, $1000$, $2000$, and $4000$ ps. Solid curves denote the mean FES, and shaded bands indicate the variance across runs. Larger $\alpha$ leads to faster convergence and lower statistical uncertainty than unbiased MD.}
    \label{fig:ala2_fes_convergence}
\end{figure}

\subsubsection{ACE-(ALA)$_2$-NME system}
We next consider the ACE-(ALA)$_2$-NME system, modeled with the Amber99 force field and solvated in $1113$ explicit TIP3P water molecules in a cubic box of size $3.23\,\mathrm{nm} \times 3.23\,\mathrm{nm} \times 3.23\,\mathrm{nm}$. Simulations are carried out in the NVT ensemble at $T = 300\,\mathrm{K}$ with time step $0.001\,\mathrm{ps}$. Short-range van der Waals interactions are truncated at $0.9\,\mathrm{nm}$, and electrostatics are treated by PME with real-space cutoff $0.9\,\mathrm{nm}$. This system probes a four-dimensional torsional CV space and therefore provides a more stringent scalability test than Ala2.

The CVs are the backbone torsions $(\phi_1,\psi_1,\phi_2,\psi_2)$. Representative time series for $\phi_1$, $\phi_2$, and $\psi_1$ are shown in Figure~\ref{fig:ala4_time_seies}. Without bias, the dynamics remains trapped for long periods in a small number of metastable regions. Once the adaptive bias is activated, transitions become substantially more frequent and the angular space is explored more broadly, especially for $\alpha=2.0$.

Figure~\ref{fig:ala4_1d_marginal} summarizes the convergence of the one-dimensional marginals along the four torsional coordinates. The major minima and barrier locations are already well resolved after approximately $4000$ ps, while extending the production run to $8000$ ps further smooths the profiles and reduces uncertainty. The corresponding two-dimensional marginals in Figure~\ref{fig:ala4_2d_marginal} reveal multiple metastable basins and well-defined transition channels in representative CV pairs. Taken together, these results show that the tensor-based density estimator remains effective when the bias is constructed in a four-dimensional CV space.

\begin{figure}
    \centering
    \includegraphics[width=0.95\linewidth]{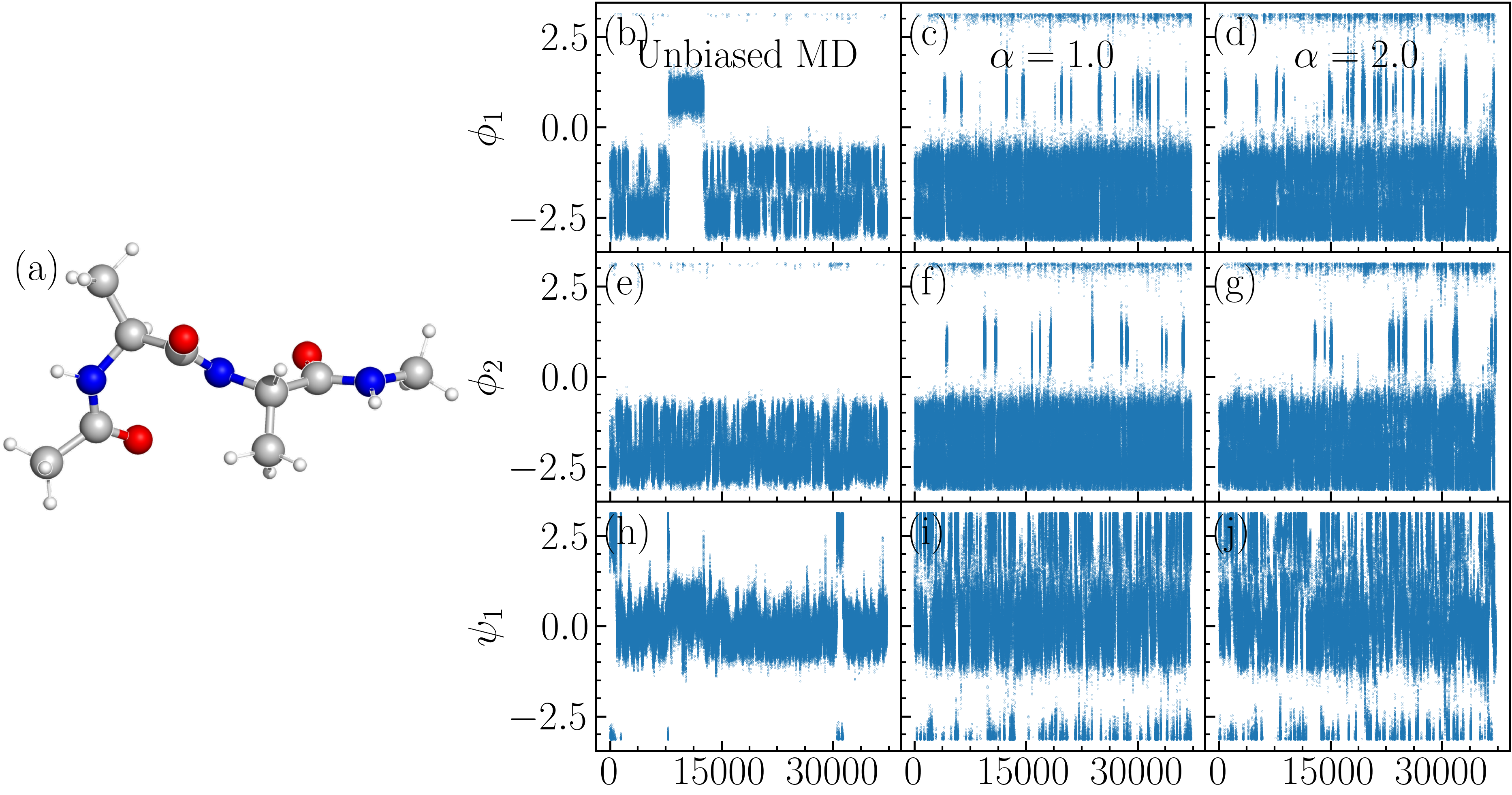}
    \caption{ACE-(ALA)$_2$-NME benchmark with 3 CVs. (a) Molecular structure of the system. The remaining panels show time series of three representative torsions under unbiased MD and biased dynamics. Columns correspond to the unbiased trajectory and to biased simulations with $\alpha=1.0$ and $\alpha=2.0$, while rows correspond to $\phi_1$ (top), $\phi_2$ (middle), and $\psi_1$ (bottom). The adaptive bias produces substantially more frequent transitions and broader angular coverage.}
    \label{fig:ala4_time_seies}
\end{figure}

\begin{figure}
    \centering
    \includegraphics[width=0.95\linewidth]{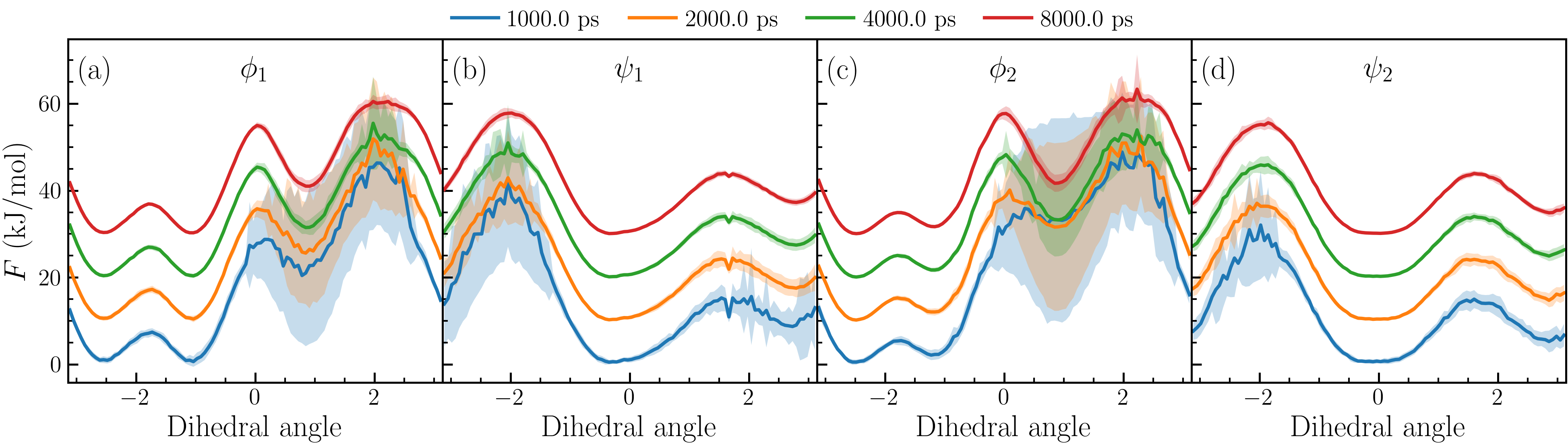}
    \caption{One-dimensional marginal free-energy profiles for ACE-(ALA)$_2$-NME at different production lengths. Panels (a)--(d) show $F(\phi_1)$, $F(\psi_1)$, $F(\phi_2)$, and $F(\psi_2)$, respectively. Curves correspond to production lengths of $1000$, $2000$, $4000$, and $8000$ ps, and the shaded bands indicate the associated variance across independent runs. The principal features of the FES are already recovered after about $4000$ ps.}
    \label{fig:ala4_1d_marginal}
\end{figure}

\begin{figure}
    \centering
    \includegraphics[width=0.8\linewidth]{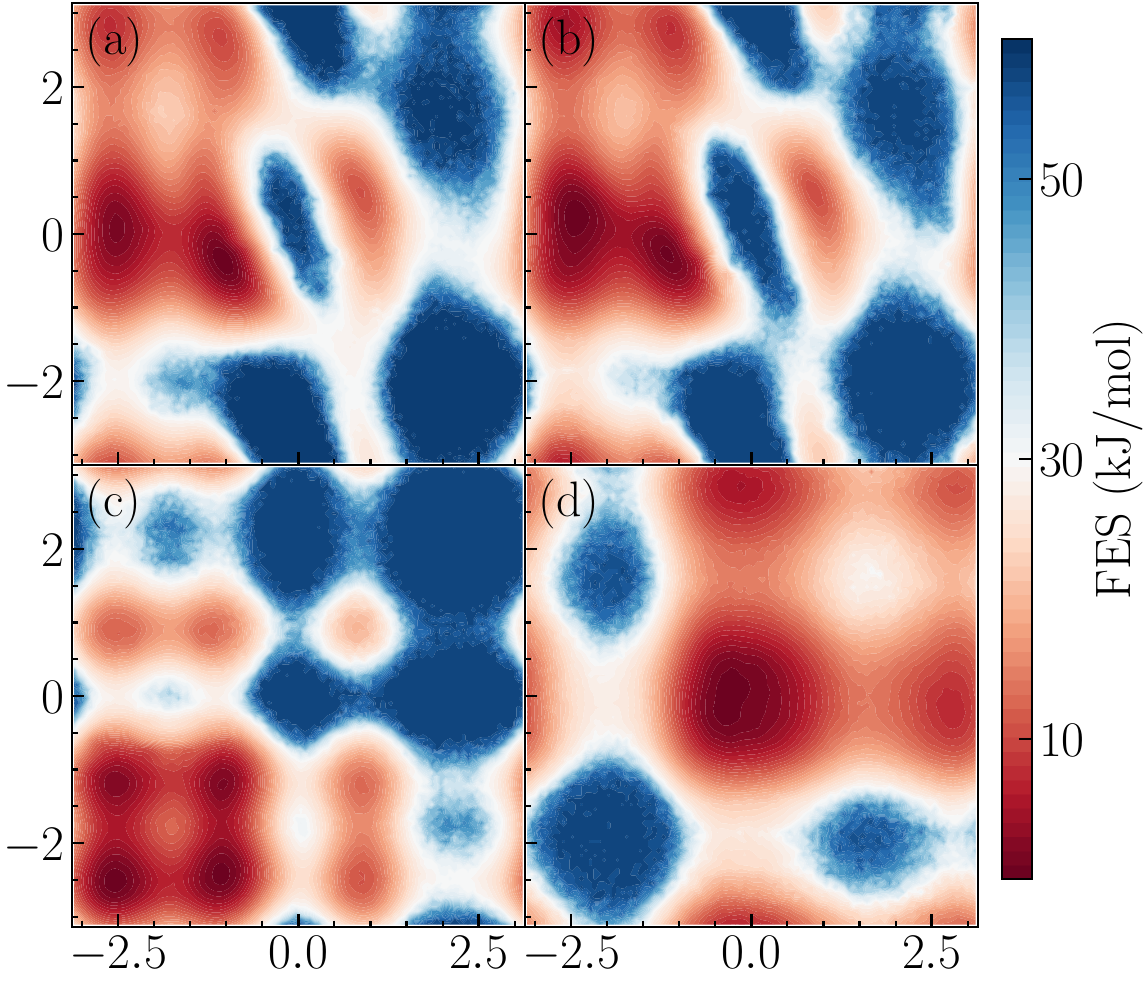}
    \caption{Two-dimensional marginal free-energy surfaces for representative pairs of torsional coordinates in ACE-(ALA)$_2$-NME. Panels (a)--(d) correspond to $(\phi_1,\psi_1)$, $(\phi_2,\psi_2)$, $(\phi_1,\phi_2)$, and $(\psi_1,\psi_2)$, respectively. The surfaces resolve the dominant metastable basins and the principal transition channels in the reduced CV space.}
    \label{fig:ala4_2d_marginal}
\end{figure}

\subsection{Peptoid systems}
\subsubsection{s-(1)-phenylethyl peptoid}
We next consider the peptoid system s-(1)-phenylethyl (s1pe), taken from~\cite{weiser2019cgenff}. The molecule is solvated in $546$ TIP3P water molecules and simulated in the NVT ensemble at $T = 300$ K. All remaining MD parameters are identical to those used for Ala2. We use the three torsion angles $\omega(C_{\alpha},C,N,C_{\alpha})$, $\phi(C,N,C_{\alpha},C)$, and $\psi(N,C_{\alpha},C,N)$ as CVs, which yields a three-dimensional biasing space.

Figure~\ref{fig:chi1_explore} examines how the bias strength affects exploration of the torsional landscape. In the unbiased run, and also for the weak bias $\alpha=2.0$, the dynamics remains trapped near a small subset of metastable states for long periods. Increasing the bias strength to $\alpha=4.0$ and $8.0$ leads to repeated switches between the major $\omega$ states and substantially broader coverage of the angular domain. The empirical marginal density correspondingly becomes less concentrated, which is consistent with improved exploration.

After $50{,}000$ ps of adaptive sampling, the bias is frozen and additional production runs are carried out for reweighting. Figure~\ref{fig:chi1_fes} shows the resulting one-dimensional marginals $F(\omega)$, $F(\phi)$, and $F(\psi)$. The reported mean and variance are computed from $32$ independent trajectories of length $50{,}000$ ps. The resulting profiles are smooth and reproducible, with narrow uncertainty bands except near high-free-energy regions, indicating that the method remains stable when the bias is constructed in a three-dimensional CV space.

\begin{figure}
    \centering
    \includegraphics[width=0.95\linewidth]{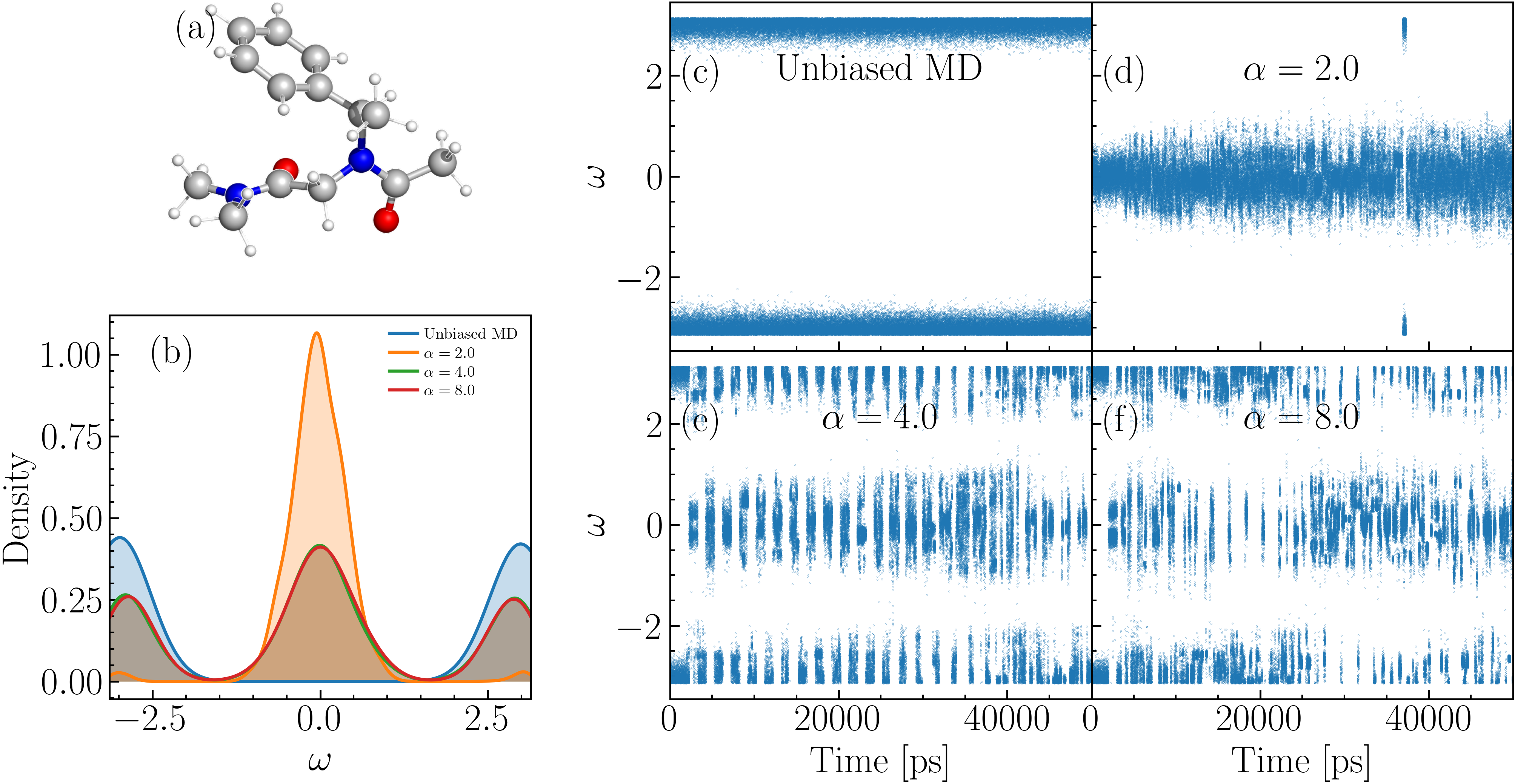}
    \caption{s-(1)-phenylethyl peptoid benchmark with 3 CVs. (a) Molecular structure of s1pe. (b) Estimated marginal density of $\omega$ obtained from trajectories generated with different bias strengths after $50{,}000$ ps. (c)--(f) Time series of $\omega$ for unbiased MD and for biased dynamics with $\alpha=2.0$, $4.0$, and $8.0$, respectively. Larger $\alpha$ yields more frequent transitions and broader exploration of the torsional space.}
    \label{fig:chi1_explore}
\end{figure}

\begin{figure}
    \centering
    \includegraphics[width=0.95\linewidth]{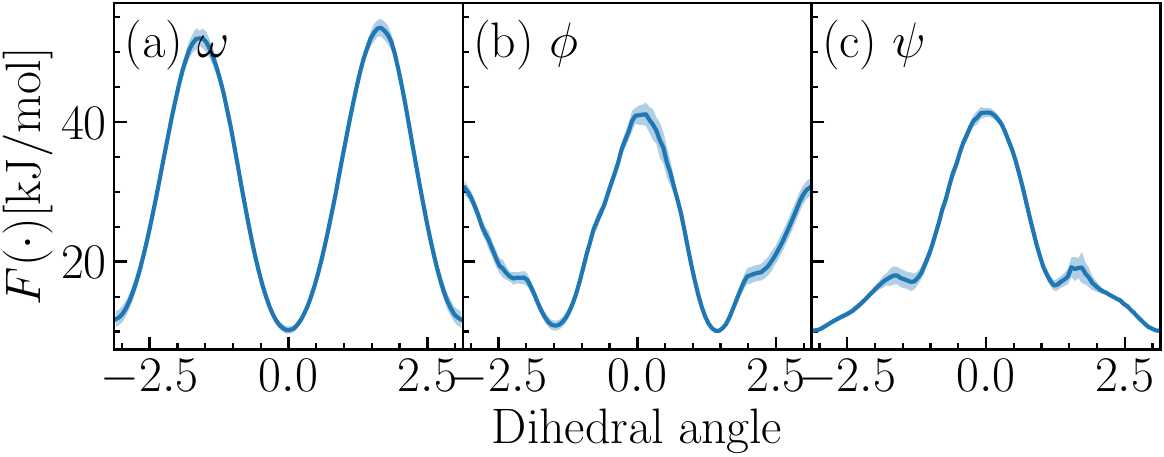}
    \caption{Reweighted one-dimensional marginal free-energy profiles for s-(1)-phenylethyl peptoid after freezing the bias at $50{,}000$ ps. Panels (a)--(c) show $F(\omega)$, $F(\phi)$, and $F(\psi)$, respectively. Solid curves denote the mean over $32$ independent production runs, and shaded bands indicate the corresponding variance.}
    \label{fig:chi1_fes}
\end{figure}

\subsubsection{Peptoid trimer $(\mathrm{s1pe})_3$}
As a higher-dimensional test, we consider the peptoid trimer $(\mathrm{s1pe})_3$ and use the nine torsional variables $\omega_1,\phi_1,\psi_1,\omega_2,\phi_2,\psi_2,\omega_3,\phi_3,\psi_3$ as CVs. This example probes the regime in which conventional grid-based biasing strategies become prohibitively expensive.

Figure~\ref{fig:chi3_explore} shows representative biased trajectories for all nine torsional coordinates. Frequent switches between multiple angular states are observed across the entire set of CVs, indicating that the adaptive bias continues to promote barrier crossing even in a nine-dimensional CV space. The corresponding reweighted one-dimensional marginals are shown in Figure~\ref{fig:chi3_fes}. Each torsion exhibits a structured, multimodal free-energy profile with clearly resolved minima and barriers. The absence of spurious oscillations in these marginals supports the robustness of the FHT density estimator in the moderate-to-high-dimensional regime.

\begin{figure}
    \centering
    \includegraphics[width=0.95\linewidth]{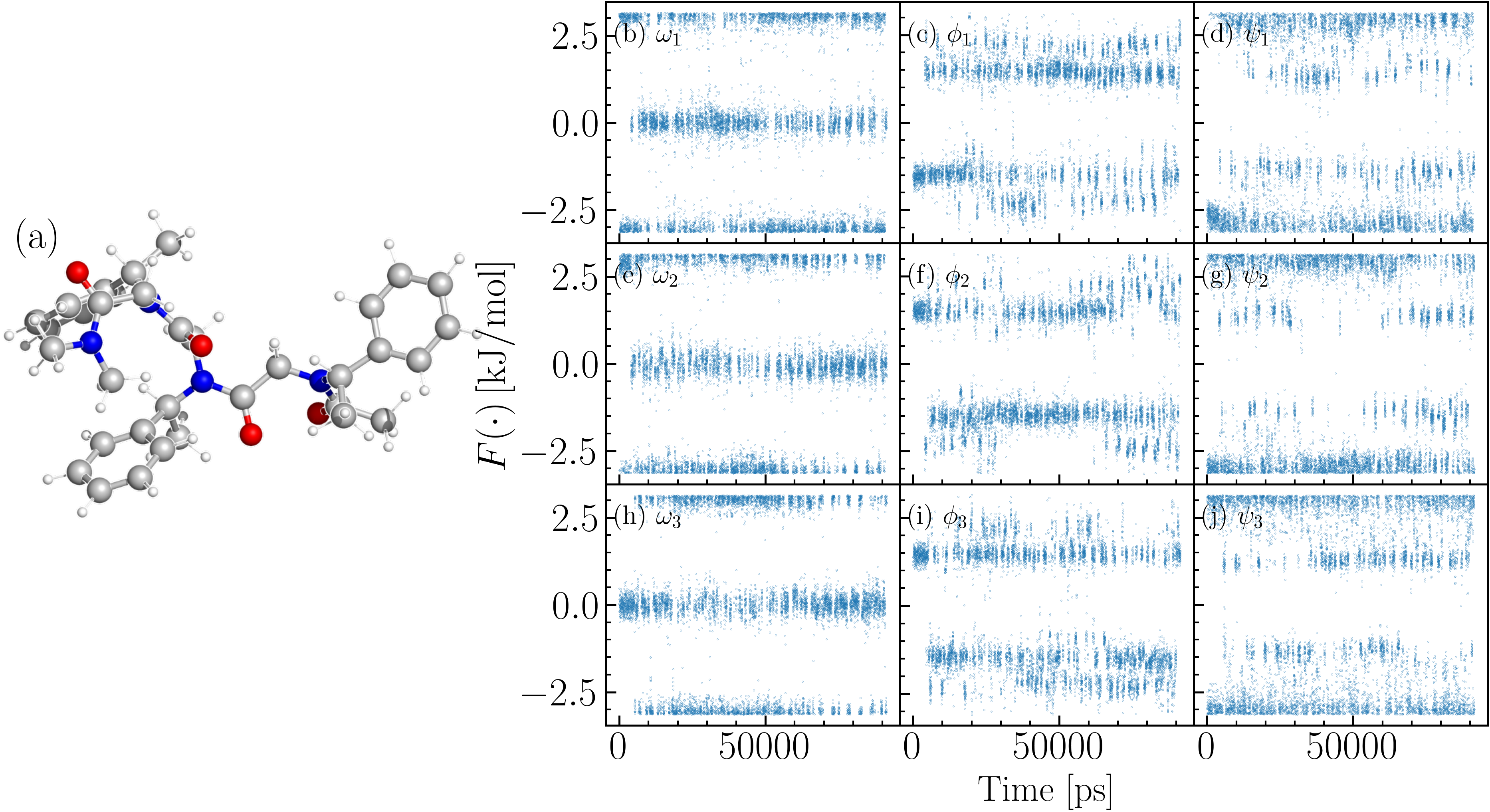}
    \caption{Peptoid trimer benchmark with 9 CVs. (a) Molecular structure of $(\mathrm{s1pe})_3$. (b)--(j) Representative time series of the nine torsional CVs $\omega_1$, $\phi_1$, $\psi_1$, $\omega_2$, $\phi_2$, $\psi_2$, $\omega_3$, $\phi_3$, and $\psi_3$ during biased sampling. Frequent transitions are observed along all coordinates, indicating sustained exploration of the high-dimensional torsional landscape.}
    \label{fig:chi3_explore}
\end{figure}

\begin{figure}
    \centering
\includegraphics[width=0.95\linewidth]{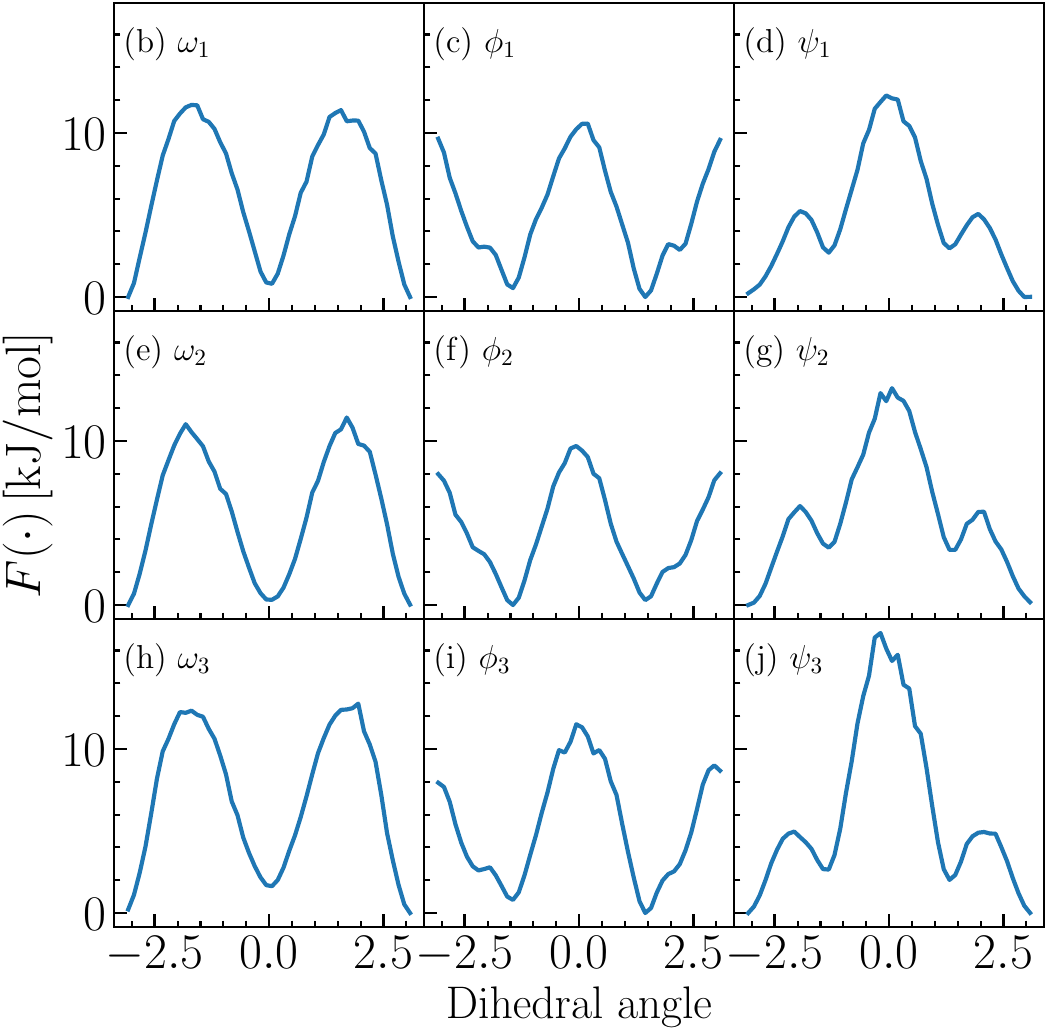}
    \caption{Reweighted one-dimensional marginal free-energy profiles for the peptoid trimer $(\mathrm{s1pe})_3$. Panels (b)--(j) correspond to $F(\omega_1)$, $F(\phi_1)$, $F(\psi_1)$, $F(\omega_2)$, $F(\phi_2)$, $F(\psi_2)$, $F(\omega_3)$, $F(\phi_3)$, and $F(\psi_3)$, respectively. The multimodal structure of these marginals confirms that the method can recover informative free-energy profiles even when the bias is built in a nine-dimensional CV space.}
    \label{fig:chi3_fes}
\end{figure}

\subsection{Protein benchmarks}
\subsubsection{Chignolin (1UAO)}
\label{subsubsec:1uao}

We next consider Chignolin (PDB: 1UAO), a designed 10-residue peptide
that folds into a $\beta$-hairpin structure and is widely used as a benchmark
for peptide folding simulations.
This system presents a substantially more challenging application than the model peptides considered above, as it exhibits a complex folding landscape with multiple metastable states ranging from the native fold to partially and fully unfolded conformations.
The molecule is solvated in $1645$ TIP3P water molecules with $7$ Na$^+$ and $5$ Cl$^-$ counterions to neutralize the system.
MD simulations are performed in the NVT ensemble at $T = 300$\,K using the Amber99-SB force field~\cite{hornak2006comparison} with a time step of $\Delta t = 0.0001$\,ps; all other simulation parameters follow the Ala2 setup.
The 16 backbone dihedral angles serve as collective variables, defining a 16-dimensional CV space.
The FHT density estimate uses a Gaussian basis with $p = 21$ basis functions per coordinate, width $\sigma = 0.1$, and maximal internal bond dimension $r = 15$.
The Softplus regularization parameters are set to $\epsilon = 1$ and $\tau = 5$, with bias strength $\alpha = 2$.

Figure~\ref{fig:1uao_exploration} illustrates how the biased dynamics progressively drives the system away from the native fold and into under-sampled regions of configuration space.
Panel~(a) shows the radius of gyration $R_g$ along a single-walker trajectory, colored by iteration number.
During the early iterations the walker remains near the compact native state ($R_g \approx 5$\,\AA), but as the bias strengthens the trajectory begins to visit increasingly extended conformations, with $R_g$ exceeding $8$\,\AA\ in later iterations.
Panel~(b) summarizes the $R_g$ distribution across all walkers as a function of iteration: both the median and the interquartile range shift steadily upward, confirming that the bias systematically promotes exploration of non-native conformations.
Panels~(c) and~(d) contrast within-iteration $R_g$ traces at iteration~5 and iteration~50, respectively.
At iteration~5 the walkers remain tightly clustered around the folded state, whereas by iteration~50 they span a broad range of $R_g$ values and undergo frequent transitions between compact and extended conformations, demonstrating that the evolved bias effectively flattens the free-energy barriers separating metastable basins.

To further characterize the resulting free-energy landscape, Figure~\ref{fig:1uao_fes} shows two-dimensional FESs reconstructed from the fixed-bias production trajectories using the multistate Bennett acceptance ratio (MBAR) estimator~\cite{shirts2008statistically}.
Panel~(a) displays the native $\beta$-hairpin structure of Chignolin for reference.
In the $(\mathrm{RMSD}_{C_\alpha}, R_g)$ projection (panel~b), the FES resolves a compact folded basin~$F$ at low RMSD and small $R_g$, an intermediate basin~$I$ at moderate RMSD, and two unfolded states $U_1$ and $U_2$ of progressively larger radius of gyration.
Representative structures extracted from each basin confirm the structural interpretation: $F$ corresponds to the native $\beta$-hairpin, $I$ to a partially disrupted turn, and $U_1$/$U_2$ to increasingly extended backbone conformations.
The $(d_1, d_2)$ projection (panel~c), where $d_1$ and $d_2$ denote the distances between the hydrogen-bonding atom pairs Asp3\,N--Gly7\,O and Asp3\,N--Thr8\,O, respectively, provides a complementary view consistent with previous studies of Chignolin folding~\cite{satoh2006folding} and reveals an extended diagonal low-free-energy channel connecting the principal metastable basins.
These results demonstrate that the proposed method can resolve multiple structurally distinct metastable states in a realistic molecular system while maintaining the statistical quality needed for reliable MBAR-based free-energy reconstruction.

\begin{figure}[t]
    \centering
    \includegraphics[width=0.9\linewidth]{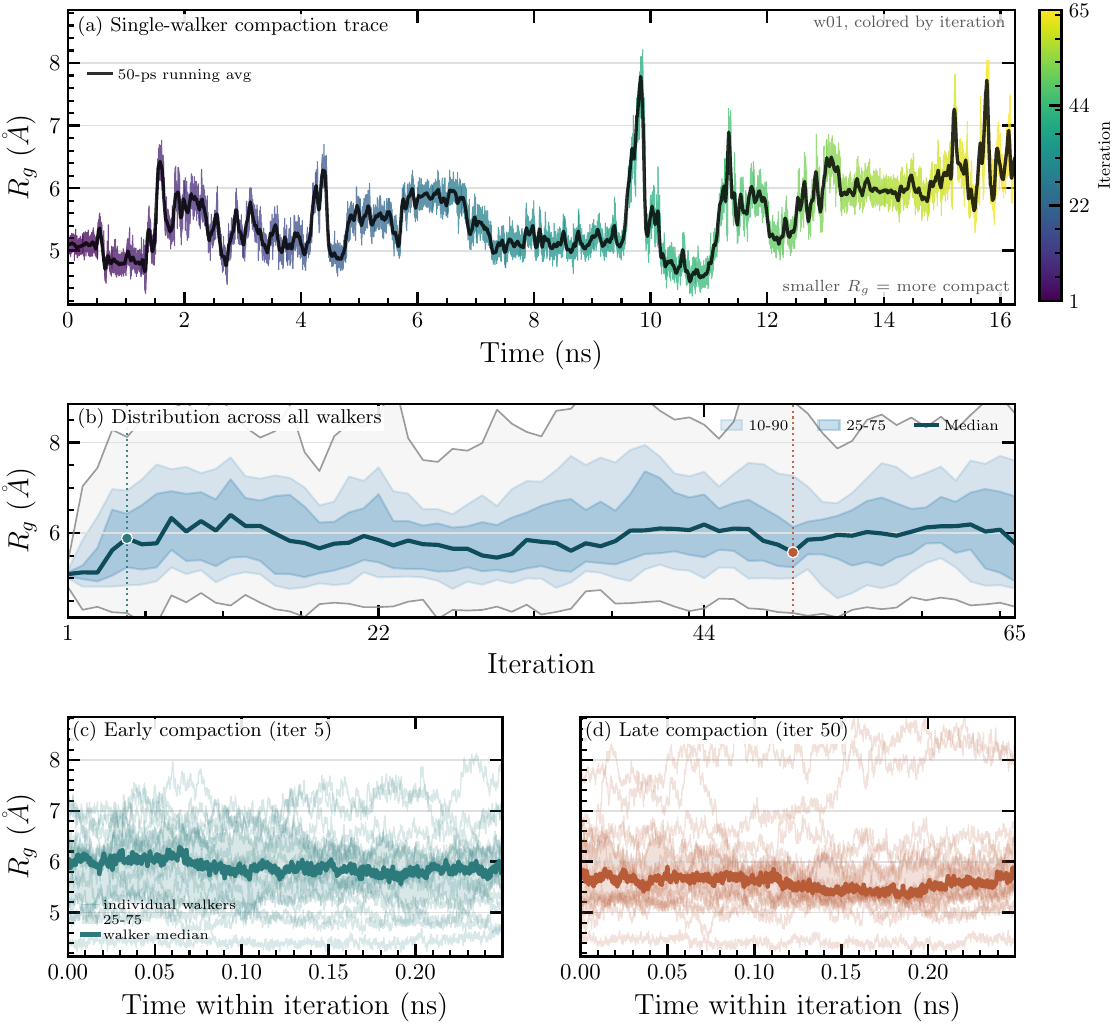}
    \caption{Configurational exploration of Chignolin (1UAO) with 16 CVs, measured by the radius of gyration $R_g$.
    (a)~$R_g$ trajectory of a single walker colored by iteration number, with the 50-ps running average shown in black.
    (b)~$R_g$ distribution across all walkers as a function of iteration; shaded bands show the 10--90th and 25--75th percentile ranges, and the solid line shows the median.
    (c,\,d)~Within-iteration $R_g$ traces for all walkers at iteration~5 (early stage) and iteration~50 (bias-strengthened stage), illustrating the progressive shift toward more extended conformations as the bias evolves.}
    \label{fig:1uao_exploration}
\end{figure}

\begin{figure}[t]
    \centering
    \includegraphics[width=0.9\linewidth]{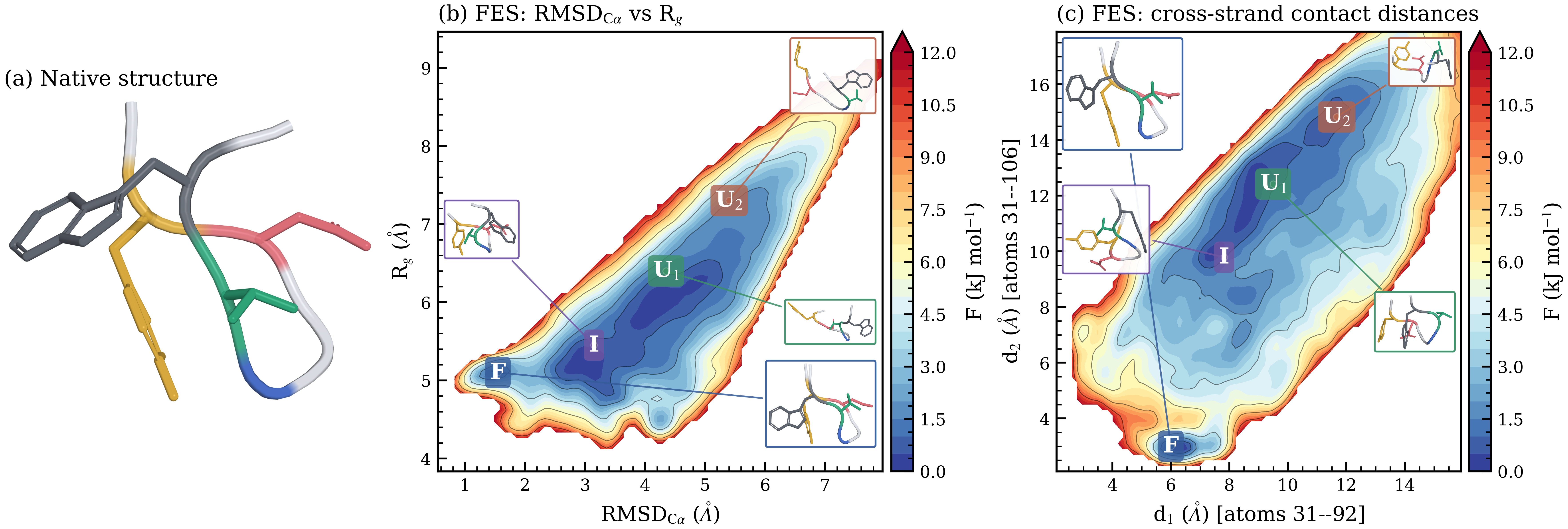}
    \caption{Free energy landscape of Chignolin (1UAO) estimated by the proposed method.
    (a)~Native $\beta$-hairpin structure.
    (b)~Two-dimensional FES projected onto $\mathrm{RMSD}_{C_\alpha}$ vs.\ $R_g$, revealing four metastable states: the folded state (F, low RMSD and compact), an intermediate (I), and two unfolded states ($U_1$, $U_2$) of increasing radius of gyration.
    (c)~FES projected onto the cross-strand contact distances $d_1$ (atoms 31--92) vs.\ $d_2$ (atoms 31--106), providing complementary resolution of the unfolded ensemble.
    Inset snapshots show representative structures from each basin.}
    \label{fig:1uao_fes}
\end{figure}

\subsubsection{Villin headpiece (1VII)}
\label{subsubsec:1vii}

To test the scalability of the method to larger proteins, we consider the villin headpiece subdomain (PDB: 1VII), a 36-residue three-helix bundle that is one of the smallest autonomously folding proteins and a longstanding benchmark for folding simulations.
Compared with the 10-residue Chignolin, this system poses a considerably greater challenge: the folding landscape involves the cooperative packing of three helices, resulting in a higher-dimensional CV space and a richer set of partially folded intermediates.
The molecule is solvated in $4395$ TIP3P water molecules with $13$ Na$^+$ and $15$ Cl$^-$ counterions to neutralize the system.
MD simulations are performed in the NVT ensemble at $T = 300$\,K using the Amber99-SB force field~\cite{hornak2006comparison} with a time step of $\Delta t = 0.0001$\,ps; all other simulation parameters follow the Ala2 setup.
The backbone dihedral angles are used as collective variables, defining a 64-dimensional CV space.
The FHT density estimate uses a Gaussian basis with $p = 21$ basis functions per coordinate, width $\sigma = 0.1$, and maximal internal bond dimension $r = 15$.
The Softplus regularization parameters are set to $\epsilon = 1$ and $\tau = 10$, with bias strength $\alpha = 2$.

Figure~\ref{fig:1vii_fes} shows two-dimensional FESs reconstructed from the fixed-bias production trajectories using MBAR~\cite{shirts2008statistically}.
Panel~(a) displays the native three-helix bundle structure.
In the $(R_g, d_{\mathrm{e2e}})$ projection (panel~b), the FES reveals a compact folded basin~$F$ at small $R_g$ and short end-to-end distance, an intermediate basin~$I$, and two unfolded states $U_1$ and $U_2$ characterized by progressively larger end-to-end distances and radii of gyration.
The $(\mathrm{RMSD}_{C_\alpha}, d_{\mathrm{e2e}})$ projection (panel~c) provides a complementary view that clearly separates the folded and intermediate basins, which overlap in $R_g$ but are distinguished by their backbone RMSD.
Inset structures extracted from each basin confirm the structural assignments.
The clear resolution of four distinct metastable states in a 64-dimensional CV space demonstrates that the proposed method scales effectively to realistic helical proteins with complex folding landscapes.

\begin{figure}[t]
    \centering
    \includegraphics[width=0.9\linewidth]{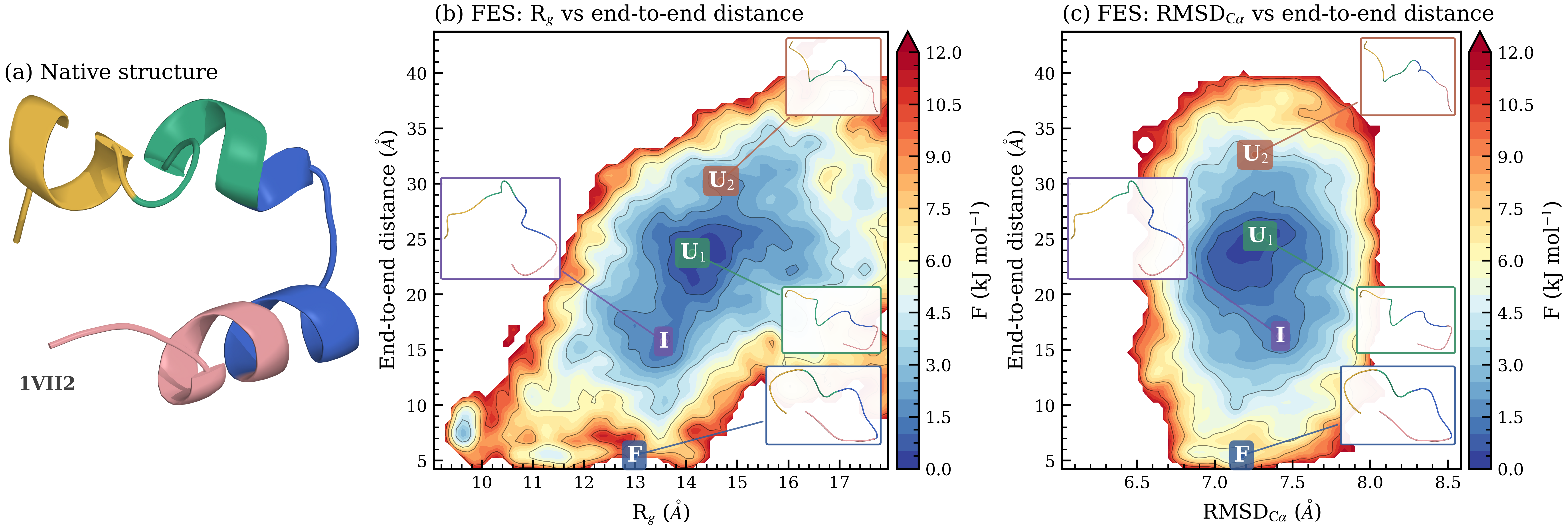}
    \caption{Free energy landscape of the villin headpiece subdomain (1VII) with 64 CVs.
    (a)~Native three-helix bundle structure.
    (b)~Two-dimensional FES projected onto $R_g$ vs.\ end-to-end distance, revealing four metastable states: the folded state (F), an intermediate (I), and two unfolded states ($U_1$, $U_2$) with progressively larger end-to-end distances.
    (c)~FES projected onto $\mathrm{RMSD}_{C_\alpha}$ vs.\ end-to-end distance, providing complementary resolution of the folded and intermediate basins.
    Inset snapshots show representative structures from each basin.}
    \label{fig:1vii_fes}
\end{figure}

\FloatBarrier
\section{Discussion}
\label{sec:discussion}

In this work, we propose a scalable adaptive enhanced-sampling framework for systems with moderate- to high-dimensional CV spaces. 
The central idea is to reformulate adaptive biasing at the level of the stochastic dynamics, so that the biasing drift is defined through a regularized path-dependent CV marginal rather than an instantaneous mean-field law. 
The formulation is motivated by the Wasserstein gradient flow interpretation \cite{lelievre2025convergence} of adaptive biasing, but introduces two modifications that are important for efficient numerical implementation. 
First, instead of regularizing the free-energy functional, we directly regularize the CV marginal entering the biasing drift. 
This modification sacrifices the exact Wasserstein gradient-flow structure, but yields a smooth and well-defined score field and avoids the computationally expensive outer convolution over the high-dimensional CV space. 
Second, we replace the instantaneous law with a weighted history measure accumulated along the trajectories.
This modification mitigates the sampling noise associated with the small-replica regime and is particularly useful for real applications like molecular simulations with a limited number of independent replicas.  The resulting dynamics leads to a path-distribution-dependent SDE.  We establish well-posedness of the regularized and path-dependent stochastic dynamics, and relate the long-time behavior to the invariant measure of the corresponding instantaneous-law dynamics under suitable dissipativity and moment assumptions. The corresponding CV marginal is approximated from accumulated trajectory samples using an optimization-free FHT approximation, which provides an efficient biasing scheme particularly suited for high-dimensional CV spaces.

The numerical experiments demonstrate that the proposed formulation can construct effective adaptive biases in CV spaces substantially beyond the dimensionality typically addressed by enhanced sampling methods for MD systems. Across benchmark problems ranging from low-dimensional potentials to complex protein systems, the adaptive dynamics facilitates the exploration of metastable regions and improves barrier crossing compared with unbiased simulations. 
In particular, the method is applied to molecular systems with CV dimensions up to \(64\), with reweighted low-dimensional FES projections resolving folded, intermediate, and unfolded basins in the protein examples. 
These results illustrate the scalability of the proposed density-based biasing strategy.
We emphasize, however, that the FHT density approximation used during the adaptive stage is not intended to serve as a final unbiased FES estimator. 
Rather, it provides a smooth operational surrogate of the history-averaged CV marginal for constructing the biasing drift. 
Unbiased free-energy information is subsequently recovered from fixed-bias production samples by reweighting, or by MBAR when samples from multiple bias stages are combined.

A closely related recent development is tensor-compressed metadynamics, named TT-metadynamics~\cite{strand2026adaptive}. This method remains within the metadynamics framework; Gaussian hills are deposited along the trajectory, and the accumulated bias potential is periodically compressed into a tensor-train representation. In this sense, the tensor representation primarily serves as a compression method of the accumulated Gaussian hills. In contrast, the present method does not deposit hills. Rather, it directly estimates the CV marginal density from accumulated trajectory data and derives the biasing drift from a regularized score field within a path-dependent McKean--Vlasov formulation. Thus, the low-rank tensor representation serves as the density estimator to close the stochastic dynamics rather than a compression device of an accumulated Gaussian hills. The resulting dynamics is better viewed as a path-dependent regularized score-driven dynamics rather than a tensorized version of metadynamics.

For future work, we note that the present regularized formulation sacrifices the exact Wasserstein gradient-flow structure, and therefore, the variational convergence theory available for the idealized adaptive-biasing dynamics does not directly apply. 
Although existing results for self-interacting McKean--Vlasov processes provide theoretical support for the use of history measures under dissipativity assumptions, a convergence theory that covers realistic molecular potentials and finite-replica simulations remains open. 
On the numerical side, the accuracy and stability of the tensor-based density estimator further depend on the effective low-rank structure of the CV marginal, the choice of basis functions, and the regularization parameters. This calls for further investigations on the adaptive selection of history weights and the tensor ranks, as well as non-negative representation (e.g., see \cite{Tang_Dwaraknath_arxiv_2025}). Also, it would be valuable to integrate the present enhanced sampling method with data-driven learning of CVs for complex systems.

\section*{Acknowledgments}
The work is supported in part by the
National Science Foundation under Grant DMS-2110981 and the
ACCESS program through allocation MTH210005.

\section*{Code availability}
A reference implementation of the proposed method, together with scripts to reproduce the numerical experiments reported in this paper, is openly available at \url{https://github.com/Lyuliyao/fht-sampler}.


\appendix

\section{Proof of Theorem~\ref{thm:path_wellposed}}\label{app:proof_wellposed}

\begin{proof}
Set $\bX_t^0=\bx$ for all $t\in [0,T]$. We define $\bX^{k+1}_t$ inductively from $\bX^{k}_t$ by the following equation:
\begin{equation}\label{equ:picard}
   \bX_t^{k+1} = \bx  + \int_0^t \bb(\bX_s^{k+1},\mu_{s}^{k,q})\intd s  + \sigma \int_0^t \intd \bW_s.
\end{equation}
At step $k+1$, we treat $\mu_{s}^{k,q}$ as a known $\mathcal P_2$-valued process. Then, by Lemma~\ref{lem:lip_cont},
\[
\|\bb(\bX_1,\mu_{s}^{k,q})-\bb(\bX_2,\mu_{s}^{k,q})\|\leq C \|\bX_1-\bX_2\|, \qquad
\|\bb(\bX,\mu_{s}^{k,q})\|\leq C(1+\|\bX\|).
\]
By Theorem~3.1.1 in~\cite{prevot2007concise}, there is a unique strong solution $\bX^{k+1}_t$ to~\eqref{equ:picard} with
\[
\mathbb E\left[\sup_{0<t<T}\left\|\bX^{k+1}_t\right\|^2\right]< \infty,
\]
which implies $\mu_{s}^{k+1,q}\in \mathcal P_2(\mathbb R^{d})$. By induction, every Picard iterate is well defined.

\paragraph{Step 1: Cauchy estimate.}
We show that $\{\bX^k_t\}$ is a Cauchy sequence. By It\^o's formula,
\[
\intd \left\|\bX_t^{k+1}-\bX_t^{k}\right\|^2 = 2\langle \bX_t^{k+1}-\bX_t^{k} , \bb(\bX_t^{k+1},\mu_{t}^{k,q})-\bb(\bX_t^{k},\mu_{t}^{k-1,q})\rangle\intd t.
\]
Splitting the right-hand side and applying the Lipschitz bound~\eqref{equ:lip}, we obtain
\[
\begin{aligned}
    \left\|\bX_t^{k+1}-\bX_t^{k}\right\|^2
    &\leq 2\int_0^t \left| \langle \bX_s^{k+1}-\bX_s^{k} , \bb(\bX_s^{k+1},\mu_{s}^{k,q})-\bb(\bX_s^{k},\mu_{s}^{k,q}) \rangle \right| \intd s\\
    &\quad + 2\int_0^t \left| \langle \bX_s^{k+1}-\bX_s^{k} , \bb(\bX_s^{k},\mu_{s}^{k,q})-\bb(\bX_s^{k},\mu_{s}^{k-1,q}) \rangle\right| \intd s\\
    &\leq C \int_0^t\left( \left\|\bX_s^{k+1}-\bX_s^{k}\right\|^2
    +  W_2(\mu_{s}^{k,q},\mu_{s}^{k-1,q})^2\right)\intd s.
\end{aligned}
\]
For the $W_2$ distance of the history measures, we have 
\[
W_2(\mu_s^{k,q},\mu_s^{k-1,q})^2
\le
\int_0^1 \left\|\bX_{sr}^k-\bX_{sr}^{k-1}\right\|^2 q(\intd r)
\le
\sup_{u\le s}\left\|\bX_u^k-\bX_u^{k-1}\right\|^2.
\]
Defining $\Delta_t^{k} = \sup_{s\leq t} \left\|\bX_s^{k}-\bX_s^{k-1}\right\|^2$ and combining, we get
\[
    \Delta_t^{k+1} \leq C \int_0^t\left( \Delta_s^{k+1}
    +  \Delta_s^{k} \right)\intd s.
\]
Gronwall's inequality yields
\[
\Delta_t^{k+1}\leq C\exp(Ct)\int_0^t \Delta_s^{k}\intd s,
\]
so that for any fixed $T>0$,
\begin{equation}\label{equ:fraction}
\mathbb E \left[\Delta_T^{k+1} \right] \leq  \frac{C_T^k}{k!}\,\mathbb E \left[\Delta_T^{1} \right] \to 0 \quad \text{as } k\to\infty.
\end{equation}

\paragraph{Step 2: Convergence.}
By Markov's inequality and~\eqref{equ:fraction},
\[
\begin{aligned}
 \mathbb P \left( \sup_{0\leq t\leq T} \|\bX_t^{k_1}-\bX_t^{k_2}\|>\epsilon \right)
 &\leq \frac{1}{\epsilon}\sum_{j=k_1}^{k_2-1}\mathbb E\!\left[\sup_{0\leq t\leq T} \|\bX_t^{j+1}-\bX_t^{j}\| \right]
 \leq  \frac{1}{\epsilon}\sum_{j=k_1}^{k_2-1}\sqrt{\mathbb E \left[\Delta_T^{j+1} \right]},
\end{aligned}
\]
which tends to zero as $k_1,k_2\to\infty$.
Hence $\{\bX^k_t\}$ is Cauchy in probability and converges in $L^2(\Omega;C([0,T];\mathbb R^{d}))$ to a limit $\bX_t$.
Along a subsequence, $\sup_{0<t<T}\|\bX^{k_i}_t-\bX_t| \to 0$ almost surely.
Since
\[
W_2(\mu_s^{k_i,q},\mu_s^{q})^2
\le
\sup_{0\le u\le s}\left\|\bX_u^{k_i}-\bX_u\right\|^2 \to 0,
\]
the continuity of $\bb$ and the linear growth bound $\|\bb(\bX_s^{k+1},\mu_s^{k,q})\|\leq C(1+\sup_{u\leq T}\|\bX_u^{k+1}\|)$, together with the dominated convergence theorem, imply
\[
\lim_{k\to\infty} \int_0^t \bb(\bX_s^{k+1},\mu_s^{k,q}) \intd s  = \int_0^t \bb(\bX_s ,\mu_s^{q} ) \intd s \quad \text{a.s.}
\]
Therefore, for any $t\in[0,T]$,
\[
\bX_t = \bx + \int_0^t \bb(\bX_s,\mu_s^q) \intd s + \sigma \int_0^t \intd \bW_s,
\]
which establishes existence.

\paragraph{Step 3: Uniqueness.}
If $\bX^1$, $\bX^2$ both satisfy~\eqref{equ:path_dependent_sde}, then
$W_2(\mu_{t}^{q,1},\mu_{t}^{q,2})^2\leq \sup_{u\le t} \|\bX^1_{u}-\bX^2_{u}\|^2$,
so that
\[
\sup_{r\le t} \|\bX^1_{r}-\bX^2_{r}\|^2\leq C \int_0^t \sup_{u\le s} \|\bX^1_{u}-\bX^2_{u}\|^2 \intd s.
\]
Gronwall's inequality implies $\bX^1_t = \bX^2_t$ for all $t\in [0,T]$ almost surely.
\end{proof}

\end{document}